\documentclass[11pt]{amsart}
\usepackage{graphicx,color}
\usepackage{amsmath,amsfonts,amssymb,amscd,latexsym,epsfig,hyperref}
\usepackage{epstopdf}
\usepackage{enumerate}
\usepackage[all,cmtip]{xy}
\usepackage[curve]{xypic}
\usepackage[mathscr]{euscript}
\usepackage[inline]{showlabels}

\DeclareFontFamily{OT1}{rsfs}{}
\DeclareFontShape{OT1}{rsfs}{n}{it}{<-> rsfs10}{}
\DeclareMathAlphabet{\curly}{OT1}{rsfs}{n}{it}

\makeatletter
\newcommand{\eqnum}{\refstepcounter{equation}\textup{\tagform@{\theequation}}}
\makeatother

\DeclareRobustCommand{\SkipTocEntry}[3]{}
\makeatletter
\newcommand\@dotsep{4.5}
\def\@tocline#1#2#3#4#5#6#7{\relax
  \ifnum #1>\c@tocdepth 
  \else
    \par \addpenalty\@secpenalty\addvspace{#2}%
    \begingroup \hyphenpenalty\@M
    \@ifempty{#4}{%
      \@tempdima\csname r@tocindent\number#1\endcsname\relax
    }{%
      \@tempdima#4\relax
    }%
    \parindent\z@ \leftskip#3\relax \advance\leftskip\@tempdima\relax
    \rightskip\@pnumwidth plus1em \parfillskip-\@pnumwidth
    #5\leavevmode #6\relax
    \leaders\hbox{$\m@th
      \mkern \@dotsep mu\hbox{.}\mkern \@dotsep mu$}\hfill
    \hbox to\@pnumwidth{\@tocpagenum{#7}}\par
    \nobreak
    \endgroup
  \fi}
\makeatother

\makeatletter \@addtoreset{equation}{section} \makeatother
\renewcommand{\theequation}{\thesection.\arabic{equation}}
\newtheorem{defn}[equation]{Definition}
\newtheorem{thm}[equation]{Theorem}
\newtheorem{lem}[equation]{Lemma}
\newtheorem{cor}[equation]{Corollary}
\newtheorem{prop}[equation]{Proposition}

\newenvironment{rmk}{\noindent\textbf{Remark}.}{\medskip}

\renewcommand{\div}{\operatorname{div}}
\renewcommand{\det}{\operatorname{det}}
\newcommand{\pic}{\operatorname{Pic}}
\renewcommand{\O}{\mathcal{O}}

\renewcommand{\hom}{\mathcal{H}om}
\newcommand{\ext}{\mathcal{E}xt}
\newcommand{\Hom}{\operatorname{Hom}}
\newcommand{\Ext}{\operatorname{Ext}}

\newcommand{\tr}{\operatorname{tr}}
\newcommand{\vir}{\operatorname{vir}}

\newcommand{\jac}{\operatorname{Jac}}

\newcommand{\vd}{\op{vd}}

\newcommand{\id}{\operatorname{id}}

\newcommand{\bu}{\bullet}

\newcommand\To{\longrightarrow}

\newcommand\beq[1]{\begin{equation}\label{#1}}
\newcommand\eeq{\end{equation}}
\def\ba #1\ea{\begin{align*}#1\end{align*}}
\def\bal#1\eal{\begin{align}#1\end{align}}

\newcommand\arXiv[1]{\href{http://arxiv.org/abs/#1}{arXiv:#1}}

\newcommand{\op}[1]{\operatorname{#1}}

\newcommand{\un}[1]{\underline{#1}}

\newcommand{\wt}[1]{\widetilde{#1}}

\renewcommand{\a}{\alpha}

\renewcommand{\t}{\theta}

\newcommand{\sff}[1]{\mathsf{#1}}

\newcommand{\bb}[1]{\mathbb{#1}}
\newcommand{\bff}[1]{\mathbf{#1}}
\newcommand{\cc}[1]{\mathcal{#1}}

\newcommand{\xr}[1]{\xrightarrow{#1}}
\newcommand{\hra}{\hookrightarrow}
\newcommand{\arr}[1]{\xymatrix@-.8pc{\ar[r]^-{#1} & }}
\newcommand{\onto}[1]{\xymatrix@-.8pc{\ar@{->>}[r]^-{#1} & }}
\newcommand{\into}[1]{\xymatrix@-.8pc{\ar@{^{(}->}[r]^-{#1} & }}
\newcommand{\toin}[1]{\xymatrix@-.8pc{ & \ar@{_{(}->}[l]_-{#1}}}

\title{Counting stable sheaves on singular curves and surfaces}
\author{Amin Gholampour}

\begin{document}
\maketitle
\begin{abstract} Given a quasi-projective scheme $M$ over $\bb C$ equipped with a perfect obstruction theory and a morphism to a nonsingular  quasi-projective variety $B$, we show it is possible to find an affine bundle $ \wt M/ M$ that admits a perfect obstruction theory relative to $B$.  We study the resulting virtual cycles on the fibers of $\wt M/B$ and relate them to the image of the virtual cycle $[M]^{\vir}$ under refined Gysin homomorphisms.  

Our main application is when $M$ is a moduli space of stable codimension 1 sheaves with a fixed determinant $L$ on a nonsingular projective surface or Fano threefold, $B$ is the linear system $|L|$, and the morphism to $B$ is given by taking the divisor associated to a coherent sheaf.
\end{abstract}
\tableofcontents

\section{Introduction}
Suppose $\pi \colon M\to B$ is a morphism of quasi-projective schemes over $\bb C$. A relative perfect obstruction theory (r.p.o.t) is a map in the derived category $D(M)$ $$\a \colon E^\bu\to \bb L_{M/B},$$ where $E^\bu$ is of perfect amplitude $[-1,0]$ (i.e. is represented by a 2-term complex of vector bundles in degrees $-1$ and $0$), $\bb L_{M/B}$ is the relative cotangent complex,  $h^0(\a)$ is an isomorphism, and $h^{-1}(\a)$ is surjective (see \cite{BF}). If $B=\op{Spec}(\bb C)$ then this data is called a (absolute) perfect obstruction theory (p.o.t) for $M$, and it  leads to construction of a virtual cycle $$[M]^{\vir}\in A_{\op{rk}(E^\bu)}(M).$$ Siebert's formula \cite{S} expresses  $[M]^{\vir}$ in terms of the Chern classes of $E^\bu$ and Fulton's Chern class of $M$: 
$$[M]^{\vir}=\big \{c\big(-(E^\bu)^\vee\big) \cap c_F(M)\big \}_{\op{rk}(E^\bu)}.$$

Suppose that $B$ is nonsingular. A key feature of an r.p.o.t is that its restriction to each fiber of $\pi$ over a closed point of $B$ gives a p.o.t (and hence results in a virtual cycle) over this fiber. This is the case because for any closed point $ b\in B$ with $M_b:= \pi^{-1}(b)$ the restriction of $\bb L_{M/B}$ to $M_b$ is isomorphic to the cotangent complex $\bb L_{M_b}$, and also the other requirements of an r.p.o.t are preserved under basechange.  This fact has been used to prove the deformation invariance of the invariants defined via virtual fundamental classes \cite[Prop 7.2]{BF}, such as Gromov-Witten and Donaldson-Thomas invariants. 

In this paper, we study this property of r.p.o.t's.  Given $\pi\colon M\to B$ as above in which $B$ is nonsingular and $M$ is virtually smooth i.e. is equipped with a p.o.t, we show that after perhaps passing to an affine bundle $\wt M\to M$, it is possible to deduce an r.p.o.t for $\wt M\to B$ and as a result a p.o.t for $\wt M_b$. The resulting virtual cycle $[\wt M_b]^{\vir}$ determines (via the identification of the Chow groups of a scheme and an affine bundle over it)  a class $$[M_b]^{\un{\vir}}\in A_{\vd_{M}-\dim B}(M_b),$$ where $\vd_M$ is the virtual dimension of $M$.  We use underline to remind ourselves that $[M_b]^{\underline{\vir}}$ is not necessarily resulted from a p.o.t on $M_b$ but rather from one on an affine bundle above it. Main applications of this paper rely on the following properties:
\begin{itemize}
\item Integrations against $[M_b]^{\un{\vir}}$ do not depend on the choice of $b$ as long as the integrands are pulled back from $H^*(M)$.
\item $[M_b]^{\un{\vir}}$ is simple to understand when the fiber $M_b$ is smooth.
\item One can apply a virtual localization formula to $[M_b]^{\un{\vir}}$ in the presence of $\bb C^*$ actions on $M$ and $B$ for which $\pi$ is equivariant.
\end{itemize}

Let $C$ be a nonsingular quasi-projective variety. In Theorem \ref{rest} this is carried out  more generally for a flat morphism  $i\colon C\to B$ onto a nonsingular subvariety of $B$, and the fiber product $$M_C:=C\times_B M$$ replacing $i\colon\{b\}\hra B$ and the fiber $M_b$, respectively, and a Siebert type Chern class formula is found for the resulting class $[M_C]^{\un \vir}$ (see Definition \ref{virc}). Theorem \ref{rest} also shows that $$[M_C]^{\un \vir}=i^! [M]^{\vir},$$ where $i^!$ is the product of the canonical orientations of a flat  morphism and a regular embedding (\cite[Sec 17.4]{Fu}).
In particular, $[M_C]^{\un \vir}$ is \emph{independent} of the choice of an affine bundle. The first two bullet points above follow from Theorem \ref{rest} and basic intersection theory, and the third one from Formula \eqref{loc}.

We apply this to the case where $M:=M^{\op s}(X,P)$ is the moduli space of stable codimension 1 sheaves with a fixed Hilbert polynomial $P$ on a nonsingular projective surface or Fano threefold $X$. It is well-known that $M^{\op s}(X,P)$ carries a natural perfect obstruction theory (\cite{T, HT}). We first take $B$ to be an irreducible component $\pic_L(X)$ of the Picard scheme $\pic(X)$ containing a fixed line bundle $L$, and $\pi$ to be the  natural morphism given by taking the determinant (see Section \ref{deter}): $$\det\colon M^{\op s}(X,P)\To \pic_L(X).$$ Let $M^{\op s}_L(X,P)\subset M^{\op s}(X,P)$ be the fiber of $\det$ over $L$, which is identified with the sub-moduli space of sheaves with fixed determinant $L$. 
The second morphism we study is $$\div\colon M^{\op s}_L(X,P)\To |L|,$$ where $|L|$ denotes the linear system of divisors of sections of $L$ and $
\div$ is the natural morphism defined by taking the divisor associated to a coherent sheaf  in $X$ (see Section \ref{supor}). For any divisor $D\in |L|$, let $M^{\op s}_D(P)$ be the fiber of $\div$ over $D$. It is identified with $M^{\op s}(D,P)$, the moduli space of stable sheaves on $D$ with Hilbert polynomial $P$ (Lemma \ref{D/L}). Note that $D$ could be singular or even reducible or non-reduced.  We can then study the resulting cycles  $[M^{\op s}_D(X,P)]^{\un \vir}$  when $D$ varies in $|L|$.   Siebert type Chern class formula of Theorem \ref{rest} gives Corollary \ref{GD}, which in turn implies $[M^{\op s}_D(X,P)]^{\un \vir}$ depends on $X$ only through the normal bundle $\O_D(D)$ of $D$ in $X$. 

We consider two main examples. In the first example, we consider $X=\bb P^2$ and the Hilbert polynomial $P(z)=dz+1$ (with respect to $\O(1)$) (see Section \ref{P2}). It turns out for each degree $d$ hypersurface $D\subset \bb P^2$ $$[M^{\op s}_D(P)]^{\un \vir}\in A_g(M^{\op s}_D(P)),\quad g:=(d-1)(d-2)/2.$$ If $D$ is nonsingular $M^{\op s}_D(P)\cong \jac(D)$, the Jacobian of the genus $g$ curve $D$, and $[M^{\op s}_D(P)]^{\un \vir}$ is its fundamental class. We use this to find an integral of a codimension $g$ class against $[M^{\op s}_D(P)]^{\un \vir}$ for arbitrary $D$. When $d=2d'>3$  and $g':=(d'-1)(d'-2)/2$, and $D'\subset \bb P^2$ is nonsingular of degree $d'$, in Proposition \ref{rk2D'} we mix this idea with virtual localization and find a relation among integrals over \begin{itemize} \item $M^{\op s}(D',2,2g'-1)$, the moduli space  of stable rank 2 degree $2g'-1$ vector bundles.  It  is a  nonsingular projective variety of dimension $4g'-3$. \item  $\pic^k(D')\times \op{Sym}^{2k+3d'-1}(D')$, the product of the symmetric product and degree $k$ Picard variety, where $k$ varies from $\lceil \frac{1-3d'}2 \rceil$ to $g'-1$. \end{itemize} 

In the second example, we consider $X=\bb P^3$ and the Hilbert polynomial $P(z)=z^2+2z+1-n$ (with respect to $\O(1)$), where $n$ is a nonnegative integer (see Section \ref{p3}).  It turns out for each degree $2$ hypersurface $D\subset \bb P^3$ $$[M^{\op s}_D(P)]^{\un \vir}\in A_0(M^{\op s}_D(P)).$$ If $D$ is a nonsingular quadratic then $M^{\op s}_D(P)\cong D^{[n]}$ is the Hilbert scheme of $n$ points on $D\cong \bb P^1\times \bb P^1$ and $[M^{\op s}_D(P)]^{\un \vir}$ is the top Chern class of a Carlsson-Okounkov K-theory class \cite{CO}.  We use this to find  $\deg [M^{\op s}_D(P)]^{\un \vir}$ for an arbitrary quadratic $D$ (see \eqref{degr}). Moreover, by an application of virtual localization formula we find a relation among the integrals of products of Carlsson-Okounkov type Chern classes\footnote{If $Y$ is nonsingular and projective, $\cc E$ and $\cc F$ are two coherent sheaves on $S\times Y$ flat over $S$, and $p\colon S\times Y\to S$ is the projection, we refer to $c_k(R\hom_p(\cc E,\cc F))$ as a Carlsson-Okounkov type Chern class.} over the following spaces (see Proposition \ref{p1p1})\begin{itemize}
\item  $M^{\op s}(\bb P^2, 2, -h, n)$, the moduli space of stable rank 2 torsion free sheaves on $\bb P^2$ with Chern classes $c_1=$ minus class of a line, $c_2=n$. It   is a  nonsingular projective variety of dimension $4n-4$.
\item  $(\bb P^1\times \bb P^1)^{[n]}$.
\item  $(\bb P^2)^{[k]}\times(\bb P^2)^{[n-k]}$, where $k\in \bb Z$ varies from $\lceil n/2 \rceil$ to $n$.
\end{itemize}
Similar analysis can be carried out when $\deg(D)=3$.  However, when $\deg(D)=d\ge 4$ the resulting  cycle $[M_D(P)]^{\un \vir}$ is always zero for dimension reason. The construction of this paper instead gives a 0-dimensional cycle on $M_C(P)$, where $C\subset |\O(d)|$ is a nonsingular subvariety of dimension $h^2(\O_D)$. See Proposition \ref{dge4} for a formulation of a precise statement. 

It is desirable to \emph{reduce} the p.o.t. of \cite{T, HT} (at least on some connected components of $M^{\op s}(\bb P^3,P)$) increasing its virtual dimension by $h^2(\O_D)$. This would fix the problem mentioned above (i.e. vanishing of $[M^{\op s}_D(P)]^{\un \vir}$) and allow us to construct a 0-dimensional cycle on $M_D(P)$ uniformly for any degree $d$ hypersurface $D\subset \bb P^3$. We plan to pursue this approach in a future work.

\subsection*{Acknowledgement}
I learned using affine bundles to modify obstruction theories from Richard Thomas during our work on \cite{GT1, GT2}. I would like to thank him for useful discussions and his answering all my questions.  

\subsection*{Conventions} We will work over the field of complex numbers. Given a morphism $f\colon X\to Y$ of scheme, we often use the same letter $f$ to denote its basechange by any morphism $Z\to Y$, i.e. $f\colon X\times_YZ\to Z$. We also sometimes suppress pullback maps $f^*$ on sheaves.

\section{Absolute to relative} \label{ar}

Let $M$ be a quasi-projective scheme and $\pi \colon M\to B$ be a morphism to a nonsingular quasi-projective variety $B$. We have the exact triangle 
\beq{ctg}\pi^*\Omega_B \arr{\phi }\bb L_{M}\arr{\psi} \bb L_{M/B} \arr{\chi} \pi^*\Omega_B[1] \eeq of cotangent complexes. In fact since $B$ is smooth $\bb L_B\cong \Omega_B$ is a vector bundle (in degree 0) of rank $=\dim B$.

If $\a \colon E^\bu\to \bb L_{M/B} $ is a relative perfect obstruction theory it is well-known that \beq{arel}\op{Cone}\big(E^\bu \xr{\chi \circ \a} \Omega_B[1]\big)[-1]\To \bb L_{M}\eeq is an absolute perfect obstruction theory (see e.g. \cite[Sec 3.5]{MPT}). 

We would like to reverse this operation possibly after pulling back to an affine bundle.  Suppose that \beq{theta} \t\colon F^\bu \to \bb L_{M}\eeq  is a given  perfect obstruction theory with the resulting virtual fundamental class $$[M]^{\vir}\in A_{\vd_{M}}(M), \qquad \vd_M:=\op{rk}(F^\bu).$$
As $F^{\bu}$ has perfect amplitude contained in $[-1,0]$, $\tau^{\ge -1}F^\bu\cong F^\bu$ and so we have a factorization $$\tau^{\ge -1} \t\colon F^\bu \arr{\t} \bb L_{M}\xr{\tau^{\ge -1}} \tau^{\ge -1} \bb L_{M}.$$ We start by trying  to lift the map $\tau^{\ge -1}\phi$  in 
\beq{lift}\xymatrix{&&F^\bu \ar[d]^-{\tau^{\ge-1}\t}&&\\ \pi^*\Omega_B \ar[rr]^-{\tau^{\ge -1}\phi} \ar@{-->}[rru]&&\tau^{\ge -1}\bb L_{M} \ar[rr]^-{\tau^{\ge -1}\psi}&&\tau^{\ge -1}\bb L_{M/B}}\eeq where the bottom row is obtained by truncating \eqref{ctg}.
Choose an isomorphism in $D(M)$   $$\tau^{\ge -1}\bb L_{M}\cong \{ L^{-1}\xr g L^0\}, \qquad F^\bu\cong \{ F^{-1}\xr f F^0\}$$ in which  $L^0$ is locally free. By choosing a very negative locally free resolution for $F^\bu$ \footnote{Very negativity ensures that each term of such a resolution behaves like projectives in the abelian
category of coherent sheaves.}, we may assume that the map $\tau^{\ge -1}\t$ in the derived category is in fact represented by the pair of maps $(\t^{-1},\t^0)$ in the following commutative diagram 
$$\xymatrix{ \ker f \ar@{^(->}[r] \ar@{->>}[d]^-{h^{-1}(\t)} &  F^{-1} \ar[d]^{\t^{-1}}\ar[r]^-f &F^0 \ar@{->>}[r] \ar[d]^{\t^{0}} & \Omega_{M} \ar@{=}[d]^-{h^{0}(\t)}\\ \ker g \ar@{^(->}[r] &  L^{-1} \ar[r]^-g & L^0 \ar@{->>}[r]  & \Omega_{M}.}$$  Define $K:=\ker\big(h^{-1}(\t)\big)$. Then, $f(K)=0=\t^{-1}(K)$ and so both maps $f$ and $\t^{-1}$ factor through $F^{-1}/K$ and we obtain a quasi-isomorphism $$\{F^{-1}/K\to F^{0}\} \sim \{L^{-1}\xr{ g} L^0\}.$$ The upshot is  that we may assume $L^0=F^0$, $L^{-1}=F^{-1}/K$, $\t^0=\id$ and $\t^{-1}\colon F^{-1}\to F^{-1}/K$ is the natural epimorphism. Such a presentation of a perfect obstruction theory was used before (e.g. see \cite{FG}). If we knew that that the map $\tau^{-1}\phi$ in \eqref{lift} was simply given by a map $\pi^*\Omega_B\to F^0$ then by the discussion above we would immediately get the desired lift. 

We may find an affine bundle $\rho \colon \wt{ M} \to M$, say of relative dimension $a$, such that $\wt{ M}$ is an affine scheme. To do this,  first embed $M$ as a closed subscheme of a projective space $\cc A$. Then, using Jouanolou trick \cite{J}, find an affine bundle $\rho \colon \wt{\cc A}\to \cc A$ whose total space is an affine variety, and then restrict this affine bundle to $M$. By \cite[Cor 2.5.7]{Kr}  $$\rho^*\colon A_{*}(M) \to A_{*+a}(\wt{M})$$ is an isomorphism. Since $\rho$ is smooth $\bb L_{\wt{M}/M} \cong \Omega_{\wt{M}/M}$ is a rank $a$ vector bundle in degree 0, and hence the natural exact triangle of cotangent complexes $$\rho^* \bb L_{M}\xr \kappa \bb L_{\wt{M}}\to \Omega_{\wt{M}/M}$$ is split over the affine scheme $\wt{M}$ say by means of a map $s\colon \Omega_{\wt{M}/M}\to \bb L_{\wt{M}}$. Therefore, if we define $$\wt F^\bu:=\rho^*F^\bu\oplus \Omega_{\wt{M}/M}$$ the map $\wt \t:=[\kappa \circ \rho^* \t \;\; s]$ clearly gives a perfect obstruction theory $\wt \t\colon \wt F^\bu \to  \bb L_{\wt{M}}$. If $[\wt{M}]^{\vir}$ is the resulting virtual cycle then by \cite[Prop 5.10]{BF} \beq{functor}[\wt{M}]^{\vir}=\rho^* [M]^{\vir} \in A_{\vd_{M}+a}(\wt{M}).\eeq

 The point of introducing this affine bundle is that by pulling back \eqref{lift} to $\wt{M}$ we can ensure that $\phi':=\rho^* (\tau^{\ge -1} \phi)$ is given by a genuine map of complexes given in the bottom row of the commutative diagram (note that $\rho$ is flat)
\beq{lfom}\xymatrix{ \rho^* F^{-1} \ar@{->>}[d]\ar[r]^-{\rho^* f} &\rho^* F^0  \ar@{=}[d] &&\\  \rho^* (F^{-1}/K) \ar[r] & \rho^* F^0   & & \ar[ll]_-{\phi'} \rho^* \pi^*\Omega_{B} \ar@{-->}[llu]_-{\wt \phi},}\eeq
 which shows that the resulting dashed arrow, denoted by $\wt \phi$, lifts the map $\tau ^{-1}\phi$ in Diagram \eqref{lift} after pulling it back to $\wt {M}$.



Next, define \bal \label{relob} &E^\bu:=\op{Cone}\big(\rho^*\pi^*\Omega_B \To \wt F^\bu\big)\cong\\\notag  &\big \{ \rho^*\pi^*\Omega_B\oplus \rho^*F^{-1}\xr{\tiny \left[\begin{array}{cc}\wt \phi & \rho^* f \\ 0 & 0\end{array}\right]}\rho^*F^0\oplus \Omega_{\wt{M}/M}\big \}.\eal We get the commutative diagram
\beq{alpha}\xymatrix{\rho^* \pi^*\Omega_B \ar@{=}[d] \ar[r] & \wt F^\bu \ar[d]^-{\wt \t}\ar[r]& E^\bu \ar@{-->}[d]^-\a\\ \rho^*\pi^*\Omega_B \ar[r] &\bb L_{\wt {M}} \ar[r] &\bb L_{\wt {M}/B}}\eeq in which both rows are exact triangles and the dashed arrow labeled $\a$ is induced by the left square in the diagram.
 \begin{lem}
$\a\colon E^\bu \to  \bb L_{\wt{ M}/B}$ is a relative perfect obstruction theory. 
 \end{lem}
 \begin{proof}
By \eqref{relob}  $E^{\bu}$ has perfect amplitude $[-1,0]$. We know that $h^0(\wt \t)$ is an isomorphism and $h^{-1}(\wt \t)$ is surjective. Applying cohomology to \eqref{alpha}, we get the commutative diagram
$$\xymatrix{h^{-1}(\wt F^\bu) \ar@{->>}[d]^-{h^{-1}(\wt \t)} \ar@{^(->}[r]& h^{-1}(E^\bu) \ar[r] \ar[d]^-{h^{-1}(\a)}  &  \rho^*\pi^*\Omega_B \ar@{=}[d] \ar[r]& h^{0}(\wt F^\bu) \ar@{=}[d]^-{h^0(\wt \t)}  \ar@{->>}[r]& h^{0}(E^\bu)\ar[d]^-{h^0(\a)} \\ h^{-1}(\bb L_{\wt{M}}) \ar@{^(->}[r]& h^{-1}(\bb L_{\wt{M}/B}) \ar[r]&  \rho^*\pi^*\Omega_B \ar[r] &\Omega_{\wt{M}} \ar@{->>}[r]& \Omega_{\wt{M}/B} }$$ in which the rows are exact. From this diagram now it is easy to conclude that $h^0(\a)$ is an isomorphism and $h^{-1}(\a)$ is surjective, which proves the claim. 
 \end{proof}

For any nonsingular subvariety $i\colon  C \into{} B$, we denote $M_{C}:=C \times_{B}M$, and $\wt M_C:=\rho^{-1}(M_C)$. This gives the diagram with both squares Cartesian \beq{fibr}\xymatrix  @-1pc{\wt M_C \ar[d]_-{\rho} \ar[r] &\wt{M} \ar[d]^-\rho\\ M_C \ar[d]_-{\pi} \ar[r] &M \ar[d]^-\pi\\ C\ar[r]^-i & B.}\eeq
\begin{thm} i) \label{rest} $\wt M_C$ is equipped with a $C$-relative perfect obstruction theory with the resulting virtual cycle $[\wt M_C]^{\vir}$ satisfying
\ba(\rho^*)^{-1}[\wt M_C]^{\vir}&= i^![M]^{\vir}=\big \{c\big(-(F^\bu)^\vee+\pi^*N_{i(C)/B}\big) \cap c_F(M_C)\big \}_{\vd_{M_C}},\ea where $i^!$ is the refined Gysin homomorphism, $\vd_{M_C}:=\vd_{M}-\dim B+\dim C$, $N_{i(C)/B}$ is the normal bundle of $i(C)$ in $B$, and $c_F(-)$ is the Fulton's Chern class.\\
ii) More generally, the same will be true if  $i$ in Diagram \eqref{fibr} is a flat morphism from a nonsingular variety $C$ onto a nonsingular subvariety of $B$. This time, $i^!$ is interpreted the product of the canonical orientations of a flat  morphism and a regular embedding (\cite[Sec 17.4]{Fu}).

\end{thm} 
\begin{proof}
i) By \cite[Prop 7.2]{BF} the restriction $$\a|_C \colon E^\bu|_C \to \bb L_{\wt M_C/C}$$ is a relative perfect theory.
Since $C$ is nonsingular by the discussion at the beginning of this section $\wt{M}_C$ is equipped with an absolute perfect obstruction theory with the virtual tangent bundle   in K-theory is given by \ba &\rho^*(F^0)^*+T_{\wt M_C/M_C}-\rho^*(F^{-1})^*-\rho^*\pi^*T_B+\rho^*\pi^*T_C\\ &=\rho^*(F^\bu)^\vee+T_{\wt M_C/M_C}-\rho^*\pi^*N_{C/B}.\ea Moreover, the resulting virtual cycle $[\wt M_C]^{\vir}$ satisfies $$ i^! [\wt{M}]^{\vir}=[\wt M_C]^{\vir}.$$
Therefore, by \cite[Thm 6.2.b]{Fu} and \eqref{functor} $$(\rho^*)^{-1}[\wt M_C]^{\vir}=(\rho^*)^{-1} i^! [\wt{M}]^{\vir}=i^![M]^{\vir}.$$ This proves the first equality. For the second equality, by \cite[Thm 4.6]{S}, \ba &[\wt M_C]^{\vir}=\big\{c\big(\rho^*(F^\bu)^\vee-T_{\wt M_C/M_C}+\rho^*\pi^*N_{C/B}\big)|_{\wt M_C} \cap c_F(\wt M_C)\big \}_{\vd_{M_C}+a}.\ea Recall that in Jouanolou's construction one may take $\wt M_C\to M_C$ to be the restriction of an affine bundle $\wt{\cc A}\to \cc A$ in which $\cc A$ is nonsingular (in fact a projective space) containing $M_C$ as a closed subset. By \cite[Ex 4.2.6, Prop 4.2.b]{Fu} \ba c_F(\wt M_C)&=c(T_{\wt {\cc A}})|_{\wt M_C}\cap s(\wt M_C, \wt{\cc  A})\\&=c\big(\rho^* T_{\cc A}+T_{\wt {\cc A}/\cc A }\big)|_{\wt M_C} \cap \rho^* s(M_C, \cc A).\ea
Putting this into the formula above for $[\wt M_C]^{\vir}$, and using the isomorphism $T_{\wt M_C/M_C}\cong T_{\wt {\cc A}/\cc A }|_{\wt M_C}$, we get the result.

ii) This follows from Part i and  \cite[Prop 7.2]{BF} and \cite[Prop 4.2.b]{Fu}.
\end{proof}

\begin{defn} \label{virc} Suppose $i\colon C\to B$ is  is a flat morphism from a nonsingular variety $C$ onto a nonsingular subvariety of $B$ (as in Theorem \ref{rest}). Define a virtual cycle $[M_C]^{\un \vir} \in A_{\vd_{M_C}}(M_C)$ by $$[M_C]^{\un \vir}:=(\rho^*)^{-1}{[\wt M_C]^{ \vir}}.$$ By Theorem \ref{rest} $[M_C]^{\un \vir}$ is independent of the choice of the affine bundle $\rho \colon \wt M\to M$. 
\end{defn}
\begin{rmk} As mentioned above, $[M_C]^{\un \vir}$ is not necessarily resulted from a perfect obstruction theory on $M_C$ but from one on an affine bundle over it. Second formula in Theorem \ref{rest} shows that it is as if it were resulted from a a perfect obstruction theory on $M_C$ with virtual tangent bundle in K-theory the pullback of $(F^\bu)^\vee-\pi^*N_{i(C)/B}$.
\end{rmk}

The following situations are interesting special cases and arise in some of our applications.
\smallskip First of all, suppose there exists an open subset $B_0\subseteq B$ such that if $M_0:=\pi^{-1}(B_0)$,  $$\pi|_{M_0}\colon M_0\to B_0$$ is a smooth morphism.  If we take 0-cohomology of \eqref{ctg} we get the usual exact sequence \beq{exact}\pi^* \Omega_B\xr{h^0(\phi)} \Omega_{M}\xr{h^0(\psi)} \Omega_{M/B}\to 0\eeq of cotangent sheaves.

\smallskip \noindent \textbf{Situation A:}  $M$ is an integral scheme.

\smallskip \noindent \textbf{Situation B:}  For any closed point $m \in M$ with $b:=\pi(m)$ the induced map $d\pi|_m \colon T_M|_m\to T_B|_b$ between Zariski tangent spaces is surjective. Equivalently, the dual map between the cotangent spaces $\Omega_B|_b\to \Omega_{M}|_m$ is injective.

\begin{prop} (i) In Situation A,  \eqref{exact} is a short exact sequence. As a result, $\tau^{\le -1} \bb L_{M}\cong \tau^{\le -1} \bb L_{M/B}.$\\
(ii) In Situation B, the map $\wt \phi$ in Diagram \eqref{lfom} is an injection of vector bundles. 
\end{prop} 
\begin{proof} (i) It suffices to show $h^0(\phi)$ is injective.
Since by assumption  $M_0$ is smooth over $B_0$, by \cite[Thm 25.1]{M}, \eqref{exact} restricted to $M_0$ is exact from the left i.e. $h^0(\phi)|_{M_0}$ is injective. On the other hand, $\pi^*\Omega_B$ is a vector bundle and so is in particular torsion free, $M_0\subset M$ is open and dense by the integrality condition, therefore $h^0(\phi)$ is injective.

(ii) Restricting the left triangle in Diagram \ref{lfom} to the point $m$ and using the assumption that the diagonal arrow is injective, we find that $\phi|_m:\pi^*\Omega_B|_m\to \Omega_{M}|_m$ is injective and hence the claim is proven.

\end{proof}

\subsection{Virtual localization} In this Section, we combine Theorem \ref{rest} and the virtual localization formula of \cite{GP} when both $M$ and $B$ are equipped with a $\bb C^*$-action such that the morphism $\pi\colon M\to B$ is $\bb C^*$-equivariant, and the fixed subschemes $B^{\bb C^*}$ and $M^{\bb C^*}$ are proper. For simplicity, assume that $B$ has only one (nonsingular) fixed component, $$i\colon C=B^{\bb C^*} \subset B.$$ 

\begin{lem} \label{MC} The fixed set $j\colon M^{\bb C^*}\subset M$ is mapped by $\pi$ to $C$, in other words $M^{\bb C^*}=(M_C)^{\bb C^*}$.\end{lem}
\begin{proof}
By \cite[Thm 2.3]{F1} $M^{\bb C^*}$ is a subscheme with trivial $\bb C^*$-action representing the functor $F$ that sends a finite type $\bb C$-scheme $S$ with trivial $\bb C^*$-action to the set of $\bb C^*$-equivariant morphisms $S\to M$; moreover, $F$ sends a morphism $S'\to S$ of $\bb C$-schemes with trivial $\bb C^*$-actions and a $\bb C^*$-equivariant morphisms $S\to M$ to their composition $S'\to M$.  Replacing $M$ by $B$, same will be true for $B^{\bb C^*}$. 

This means any $\bb C^*$-equivariant morphisms $S\to M$ factors through $M^{\bb C^*}\hra M$. But then the composition  
$$S\to M^{\bb C^*}\hra M\xr{\pi} B$$ is also $\bb C^*$-equivariant so it must factor through $C=B^{\bb C^*} \hra B$. Thus, in particular,  scheme theoretic image $\pi(M^{\bb C^*})$ must be contained in $C$ proving the claim.
\end{proof}  

 By \cite{GP} the \emph{fixed part} of the obstruction theory $\t$ (introduced in \eqref{theta}), $$\t^{\op f} \colon F^{\bu, \op f}\To \bb L_{M^{\bb C^*}},$$ is a perfect obstruction theory and the resulting virtual cycle $ [M^{\bb C^*}]^{\vir}$ recovers $[M]^{\vir}$ by the virtual Atiyah-Bott localization formula
\beq{GP}[M]^{\vir}=j_*\frac{[M^{\bb C^*}]^{\vir}}{e\big((F^{\bu, \op m})^\vee\big)}\in A_*^{\bb C^*}(M)\otimes_{\bb Q[\bff s]} \bb Q[\bff s,\bff s^{-1}],\eeq where the superscript $\op m$ means the \emph{moving part}.

Denote the inclusion $M^{\bb C^*} \subset M_C$ by $j'$. We know that $M^{\bb C^*}\times_B C\cong M^{\bb C^*}$ by Lemma \ref{MC}, so by the self-intersection formula \cite[Cor 6.3]{Fu} and \cite[Thm 6.2]{Fu} (in combination with \cite{EG}), applying $i^!$ to both sides of \eqref{GP}, we get 
\beq{loc} [M_C]^{\un \vir}=j'_*\frac{[M^{\bb C^*}]^{\vir}}{e\big((F^{\bu, \op m})^{\vee}-\pi^*N_{C/B}\big)}\in A_*^{\bb C^*}(M_C)\otimes_{\bb Q[\bff s]} \bb Q[\bff s,\bff s^{-1}].\eeq

\section{Moduli space of sheaves with codimension 1 supports}
Let $X$ be a nonsingular projective variety of dimension $n$ with a polarization $\cc O(1)$. Fix a polynomial $P\in \bb Q[z]$ of degree $n-1$.
Let $$ M^{\op{ss}}=M^{\op{ss}}(X,P)$$ be the (coarse) moduli space of semistable sheaves with Hilbert polynomial $P$. By \cite[Thm 4.3.7]{HL} it \emph{universally co-represents} the moduli functor
$$\cc M^{\op{ss}}\colon (\op{Sch}/\bb C)^{\op{op}}\To \op{Sets}$$ that for  any finite type $\bb C$-scheme $S$, $\cc M^{\op{ss}}(S)$ is the set of $S$-flat families of Gieseker semistable sheaves with Hilbert polynomial $P$ modulo the equivalence relation \beq{equiv}\cc F\sim \cc F'\quad  \Leftrightarrow \quad \exists \;N\in \op{Pic}(S) \text{ s.t. } \cc F'\cong  \cc F \otimes p_2^* N.\eeq Moreover, for any morphism $g\colon S'\to S$ of finite type $\bb C$-schemes and any $\cc F \in \cc M^{\op{ss}}(S)$ $$\cc M^{\op{ss}}(g)(\cc F)=(\id \times g)^* \cc F.$$ $M^{\op{ss}}$ is a projective scheme and contains the moduli space of stable sheaves, denoted by \beq{mods} M^{\op s}=M^{\op s}(X,P), \eeq as an open subset. 

By \cite[Prop 2.1.10]{HL}, for any $\cc F\in \cc M^{\op{ss}}(S)$, the determinant $\det(\cc F) \in \op{Pic}(X\times S)$ is well-defined, and taking determinant commutes with basechange. As $M^{\op{ss}}$ co-represents the functor $\cc M^{\op{ss}}$, one gets a morphism 
$$\det \colon M^{\op{ss}}\To \op{Pic}(X).$$ Any pure sheaf $\cc F$ with Hilbert polynomial $P$ is supported in codimension 1, and $c_1(\cc F)=c_1(\det \cc F)$. If $D_1,\dots, D_r$ are irreducible components of the reduced support of $\cc F$ let $m_i$ be the length of stalk of $\cc F$ at the generic point of $D_i$. As defined in \cite{F, KM}, let $$\div(\cc F):=\sum{m_i}D_i$$ be the (Cartier) divisor associated to $\cc F$. This divisor has some nice properties, for example, in contrast to the scheme theoretic support of a coherent sheaf, it is well-behaved in flat families (see Lemma \ref{rho}) and also $\cc O(\div(\cc F))\cong \det(\cc F)$ and $c_1(\cc F)=\sum m_i [D_i]$.

Now fix a line bundle $L$ on $X$, and define $$M^{\op{ss}}_L:=\det^{-1}(L).$$ Because $M^{\op{ss}}$ universally co-represents $\cc M^{\op{ss}}$, $M^{\op{ss}}_L$ co-represents the subfunctor $\cc M^{\op{ss}}_L$ of $\cc M^{\op{ss}}$ that to any finite type $\bb C$-scheme $S$ assigns the set of $S$-flat families $\cc F$  of Gieseker semistable sheaves with Hilbert polynomial $P$ satisfying $\det \cc F= L \boxtimes N$ for some $N\in \pic(S)$, modulo the equivalence relation \eqref{equiv}. Let $M^{\op s}_L \subseteq M^{\op{ss}}_L$ be the open subset of stable sheaves.

\begin{lem} \label{rho}  Suppose that $\cc F$ is an $S$-flat family of pure codimension 1 coherent sheaves with fixed determinant $L$.  Then, there exists a morphism  $\rho \colon S\to |L|$ such that   $\rho(s)$ corresponds to $\op{div}(\cc F|_{X\times s})\in |L|$ for any point $s\in S$.
\end{lem}
\begin{proof}
By \cite{F, KM} $\op{div}(\cc F)\subset X\times S$ is an effective Cartier divisor and by \cite[Lem 7.3]{F} for any $s\in S$  \beq{div} \op{div}(\cc F|_{X\times s})=\op{div}(\cc F)|_{X\times s}.\eeq 
By \cite[Lecture 10 part 4]{Mu} \eqref{div} implies that $\op{div}(\cc F)\subset X\times S$ is flat over $S$. If $s\in S $ is a closed point $\cc O(\op{div}(\cc F)|_{X\times s})\cong L$. Therefore we get a canonical morphism to $|L|$ as desired.

\end{proof}

If $\cc F$ is $S$-flat family as in Lemma \ref{rho}, since $\cc F$ is a torsion sheaf on $X\times S$, from the construction of \cite{F, KM} $$\op{div}(\cc F)=\op{div}(\cc F\otimes N)$$ for any $N\in \op{Pic}(X\times S)$. In particular, by Lemma \ref{rho} we get a morphism of functors $\cc M^{\op{ss}}_L \to h_{|L|}$, which factors through a unique morphism $h_{M^{\op{ss}}_L}\to h_{|L|}$ as $M^{\op{ss}}_L$ co-represents $\cc M^{\op{ss}}_L$. By Yoneda lemma the last morphism of functors is equivalent to a morphism of schemes \beq{divm}\div \colon M^{\op{ss}}_L\To |L|.\eeq 
Let $\bb D  \subset X\times |L|$ be the universal divisor. And let $M^{\op{ss}}(\bb D/|L|,P)$ be the moduli space of semi-stable sheaves with Hilbert polynomial $P$ on the fibers of $\bb D\to |L|$ (\cite[Thm 4.3.7]{HL}). If $i\colon C\subset |L|$ is a nonsingular subvariety (or more generally a flat morphism onto a nonsingular subvariety of $|L|$), let $$M^{\op{ss}}_C(P):=M^{\op{ss}}_L(P)\times_{|L|} C,\qquad \bb D_C:=\bb D\times_{|L|} C.$$ 
Then we have an identification of moduli spaces.
\begin{lem} \label{D/L} We have an isomorphism of schemes $$M^{\op{ss}}(\bb D/|L|,P)\cong M_L^{\op{ss}}(P), \qquad M^{\op{ss}}(\bb D_C/C,P)\cong M_C^{\op{ss}}(P)$$

\end{lem}
\begin{proof} For the first isomorphism it suffices to identify the corresponding moduli functors. Let $S$ be a finite type $\bb C$-scheme. If $\cc F$ is an $S$-flat family of semistable sheaves with determinant $L$ and Hilbert polynomial $P$ on $X\times S$, as above $\div(\cc F) \subset X\times S$ is an $S$-flat effective Cartier divisor and hence determines a modular morphism $S\to |L|$. So there exists an $S$-flat family $\cc G$ of sheaves on $\div(\cc F) \cong S \times_{|L|} \bb D$ that pushes froward to $\cc F$.

Conversely, given a morphism $S\to |L|$ and an $S$-flat family $\cc G$ of sheaves with Hilbert polynomial $P$ on the fibers of  $S \times_{|L|} \bb D\to S$, the pushforward of $\cc G$ via the inclusion $$S \times_{|L|} \bb D\into{} S\times_{|L|} |L| \times X\cong S\times X$$ is an $S$-flat family of sheaves with determinant $L$ and Hilbert polynomial $P$ on $S\times X$. Finally, these two constructions are evidently inverse of each other.

The second isomorphism follows from the first one and the fact that either of moduli functors above is universally co-represented (\cite[Thm 4.3.7]{HL}).
\end{proof}

In particular, if $D \subset X$ is a divisor corresponding to a closed point $|L|$ then $$M^{\op{ss}}_D(P):=\div^{-1}(D)\cong M^{\op{ss}}(D,P)$$ the moduli space of semistable sheaves on $D$ with Hilbert polynomial $P$.

\begin{lem} \label{lns} \begin{enumerate}[(i)]
\item Let $|L|_{\op{ir}}\subset |L|$ be the open locus of reduced and irreducible divisors. Then $\div^{-1}(|L|_{\op{ir}}) \subseteq M^{\op s}_L$. 
 \item Let $|L|_{\op{ns}}\subset |L|$ be the open locus of nonsingular divisors. If $\dim X=2$ or $3$, then $\div|_{\div^{-1}(|L|_{\op{ns}})}$ is a smooth morphism.
\end{enumerate}
\end{lem}
\begin{proof}
(i) Suppose $\cc F$ is a pure sheaf with Hilbert polynomial $P$. If $D:=\div(\cc F)$ is reduced and irreducible then by definition $\cc F$ is the pushforward of a rank 1 torsion free sheaf on $D$ and hence is stable. This proves the first part. 

(ii) Suppose $D\in |L|_{\op{ns}}$ is a closed point (corresponding to a nonsingular effective divisor). By the first part $\div^{-1}(D)$ is identified with the moduli space of rank 1  torsion free (hence stable) sheaves on $D$ with Hilbert polynomial $P$. If $\dim X=2$ then $D$ is a nonsingular curve and so $\div^{-1}(D)\cong \op{Jac}(D)$, which is nonsingular. If $\dim X=3$ then $D$ is a nonsingular surface and so $\div^{-1}(D)$ admits a smooth morphism to $\op{Jac}(D)$ with fibers isomorphic to $D^{[m]}$, the Hilbert scheme of $m$ points on $D$ for some integer $m\ge 0$ (depending on $P$). Therefore $\div^{-1}(D)$  is again nonsingular.

In any case the analysis above shows that $\div|_{\div^{-1}(|L|_{\op{ns}})}$ is flat (by comparing the Hilbert polynomials of fibers over closed points) and the relative cotangent sheaf $\Omega_{\div^{-1}(|L|_{\op{ns}})/|L|_{\op{ns}}}$ is locally free (as it has a constant rank over closed points). Therefore, $\div|_{\div^{-1}(|L|_{\op{ns}})}$ is a smooth morphism proving the second part of lemma.

\end{proof}

Let   $\Ext^{i}(\cc F,\cc F)_0:=\ker\big( \Ext^{i\ge 3}(\cc F,\cc F)\xr{\op{tr}}H^i(\cc O_X)\big)$. Suppose that \beq{trfr} \Ext^{i\ge 3}(\cc F,\cc F)_0=0 \qquad \forall \text{ closed point } \cc F \in M^{\op s}.\eeq
By \cite[Sec 4.4]{HT} $M^{\op s}$ is then equipped with a perfect obstruction theory \beq{pot}\t \colon \big(\tau^{[1,2]}R\hom_p(\bb F, \bb F )\big)^\vee [-1]\To \bb L_{M^{\op s}},\eeq where $p\colon M^{\op s}\times X\to M^{\op s}$ is the projection, $\bb F$ is the universal (twisted) sheaf, and $\t$ is obtained from the Atiyah class in $\Ext^1_{M^{\op s}\times X}(\bb F, \bb F \otimes \bb L_{M^{\op s}})$. 
Note that since $\op{rk}(\cc F)=0$ for $\cc F\in M^{\op s}$, one cannot use the method of \cite[Sec 4.2]{HT} to equip $M^{\op s}_L$  with a perfect obstruction theory. One of our goal is to do this by method of Section \ref{ar} (after passing to an affine bundle).


Denote the virtual cycle of \eqref{pot} by $$[M^{\op s}]^{\vir}\in A_{\vd_{M^{\op s}}}(M^{\op s}), \qquad \vd_{M^{\op s}}:=\op{ext}^1(\cc F,\cc F)-\op{ext}^2(\cc F,\cc F).$$ By \cite{KM}, this virtual cycle is supported on the locus of $\cc F \in M^{\op s}$, where \beq{ext2}\Ext^{2}(\cc F,\cc F)\xr{\op{tr}}H^2(\cc O_X)\eeq fails to be surjective.

\subsection{Fixing determinant} \label{deter} As in Section \ref{ar},  fix an affine bundle $$\rho \colon \wt M^{\op s} \To M^{\op s}.$$ which is an affine scheme. Denote by $\wt M_L^{\op s}$ the restriction of $\wt M^{\op s}$ to $M^{\op s}_L$. Let $L$ be a line bundle $X$ and recall the definition $M^{\op s}_L= \det^{-1}(L)$. Denote by $\pic_L(X) \subset \pic(X)$ the connected component containing $L$. $\pic_L(X)$ is a nonsingular projective variety of dimension $h^{0,1}(X)$ with tangent bundle $T_{\pic_L(X)}\cong \cc O^{h^{0,1}(X)}$. Without loss of generality we replace $M^{\op s}$ by its open and closed subset mapping to $\pic_L(X)$. We can then apply Theorem \ref{rest} to the morphism $$\det \colon M^{\op s}\to \pic_L(X)$$ and $i:\{L\} \hra \pic_L(X)$:
\begin{prop} \label{det}
$\wt M^{\op s}_L$ is equipped with a perfect obstruction theory with the K-theory class of virtual tangent bundle $$-\rho^*\tau^{[1,2]}R\hom_p(\bb F, \bb F )+T_{\wt M^{\op s}_L/M^{\op s}_L}-\cc O^{h^{0,1}(X)}.$$ The virtual cycle $[M^{\op s}_L]^{\un \vir} \in A_{\vd_{M^{\op s}_L}}(M^{\op s}_L)$ (Definition \ref{virc}) satisfies
\ba [ M^{\op s}_L]^{\un \vir}&= i^![M^{\op s}]^{\vir}=\big \{c\big(R\hom_p(\bb F, \bb F )\big)\cap c_F(M^{\op s}_L)\big \}_{\vd_{M^{\op s}_L}},\ea where $\vd_{M^{\op s}_L}:=\vd_{M^{\op s}}-h^{0,1}(X)$. 
\end{prop} \qed

\subsection{Fixing support} \label{supor} We use the notation fixed in Section \ref{det}. The linear system $|L|$ is a  projective space of dimension $h^0(L)-1$ and the K-theory class of tangent bundle $T_{|L|}=\cc O_{|L|}(1)^{h^{0}(L)}-\cc O$.  Let $D\in |L|$ be a closed point. This time, we apply Theorem \ref{rest} to the morphism $$\div \colon M^{\op s}_L\to |L|$$ and $i_{|L|}:\{D\} \hra |L|$ and use Proposition \ref{det}. Let $M^{\op s}_D:=\div^{-1}(D)$ and $\wt M^{\op s}_D$ be the restriction of the affine bundle $\wt M^{\op s}$ to $M^{\op s}_D$.
\begin{prop} \label{supp}
$\wt M^{\op s}_D$ is equipped with a perfect obstruction theory with the K-theory class of virtual tangent bundle $$-\rho^*\tau^{[1,2]}R\hom_p(\bb F, \bb F )+T_{\wt M^{\op s}_D/M^{\op s}_D}-\cc O^{h^{0,1}(X)-1}-\cc O_{|L|}(1)^{h^{0}(L)}.$$ The virtual cycle $[M^{\op s}_D]^{\un \vir}$ (Definition \ref{virc}) satisfies
\ba [M^{\op s}_D]^{\un \vir}&= i_{|L|}^![M^{\op s}_L]^{\vir}=\big \{c\big(R\hom_p(\bb F, \bb F )\big) \cap c_F(M^{\op s}_D)\big \}_{\vd_{M^{\op s}_D}},\ea where $\vd_{M^{\op s}_D}:=\vd_{M^{\op s}}-h^{0,1}(X)-h^0(L)+1$. 
\end{prop} \qed

\subsection{Dimensions 2 and 3}
Our main case of interest is when $\dim X=2$ or $3$. When dimension is 2 Condition \eqref{trfr} is always satisfied, and for dimension  3 it is satisfied for example when the Hilbert polynomial of $K_X$ is less than the Hilbert polynomial of $\O_X$ (e.g. when $X$ is Fano).

Recall by Lemma \ref{lns} if $D\in |L|_{\op{ir}}$ then  $$M^{\op s}_D\cong M^{\op s}(D,P)$$ is identified with the moduli space of rank 1 torsion free sheaves on $D$, which is  a projective schemes (regardless $M^{\op s}$ is projective or not, because every rank 1 torsion free sheaf is automatically stable). When $D\in |L|_{\op{ns}}$, as shown in Lemma \ref{lns} $M^{\op s}(D,P)$, is smooth, and we will shortly compare the virtual class $[M^{\op s}_D]^{\vir}$ of Proposition \ref{supp} with the fundamental class $[M^{\op s}(D,P)]$. If $D\not \in |L|_{\op{ns}}$ much less is known about $M^{\op s}(D,P)$: for example, when $\dim X=2$ it is known that $M^{\op s}(D,P)$ (which is a compactified Jacobian) is irreducible \cite{AIK}, and when $\dim X=3$  it is only known (to our knowledge) that Hilbert schemes of points on a surface with rational double point singularities (which is an instance of $M^{\op s}(D,P)$) is irreducible \cite{Z}.

\begin{lem} \label{Rhoms} Let $i \colon \cc D \subset X\times S$ be a Cartier divisor flat over $S$,  $\cc E$ and $\cc G$ be two coherent sheaves on $\cc D$ also flat over $S$. Let $p\colon X\times S\to S$ be the projection and $p'=p \circ i\colon \cc D\to S$. Then the perfect complex $R\hom_p(i_*\cc E, i_* \cc G)$ fits in an exact triangle as follows 
$$R\hom_{p'}(\cc E,  \cc G)\to R\hom_p(i_*\cc E, i_* \cc G)\to R\hom_{p'}(\cc E, \cc G(\cc D))[-1].$$
\end{lem}
\begin{proof} This is a relative version of \cite[Lem 3.24]{T}, we will adapt the same proof here. Applying the functor $R\hom(-,i_*\cc G)$ to the natural exact sequence $0\to \cc O(-\cc D)\to \O \to i_*\O_{\cc D}\to 0$ one sees that $$R\hom(i_*\O_{\cc D}, i_* \cc G)\cong i_* \cc G\oplus i_* \cc G(\cc D)[-1]. $$ One then concludes for any locally free sheaf $E$ on $\cc D$ \beq{RhE} h^{j} \big(R\hom (i_*E, i_* \cc G)\big )\cong \begin{cases} i_* \hom(E,\cc G) & j=0 \\ i_*\hom (E,\cc G(\cc D)) & j=1\end{cases},\eeq and $h^{j} \big(R\hom (i_*E, i_* \cc G)\big )=0$ if $j\neq 0, 1$.

Now let $E^\bu \to \cc E\to 0$ be a locally free resolution on $\cc D$, and $0\to i_*\cc G\to I^\bu$ be an injective resolution on $X\times S$. The total complex of the double complex $\hom(i_*E^\bu, I^\bu)$ is quasi-isomorphic to $R\hom (i_*\cc E, i_* \cc G)$. By \eqref{RhE}, for each $j$, one can replace $\hom(i_*E^j, I^\bu)$ by its truncation $\tau^{\le 1}$ to get a quasi-isomorphic double complex $B^{j,\bu}$ with two columns fitting in an exact sequence of complexes as \beq{ihom} 0\to i_*\hom(E^j,\cc G) \to B^{j,0}\to B^{j,1}\to  i_*\hom(E^j,\cc G(\cc D))\to 0.\eeq Now  $\hom(E^j,\cc G)$, $\hom(E^j,\cc G(\cc D))$  and the total complex of $B^{j,\bu}$ are quasi-isomorphic to $i_*R\hom(\cc E,  \cc G)$,   $i_*R\hom(\cc E, \cc G(\cc D))$ and $R\hom(i_*\cc E, i_* \cc G)$, respectively. Therefore \eqref{ihom} implies an exact triangle $$i_*R\hom(\cc E,  \cc G)\to R\hom(i_*\cc E, i_* \cc G)\to i_*R\hom(\cc E, \cc G(\cc D))[-1]$$ to which we can apply $R p_*$ and so get the result.
\end{proof}

%

In Lemma \eqref{Rhoms} if we set $\cc E=\cc G$ the trace map \beq{tr3}\ext^3_p(i_* \cc G,i_*\cc G)\xr{\tr} R^3p_*\O\eeq lifts to $\tau^{\ge 1}R\hom_p(i_* \cc G,i_*\cc G)  \to R^3p_*\O[-3]$. The first arrow of the exact triangle in Lemma \ref{Rhoms} then gives a map denoted by  $$f \colon \tau^{\ge 1} R\hom_{p'}( \cc G,\cc G) \to R^3p_*\O[-3].$$

\begin{cor}\label{conf}
In the situation of Lemma \ref{Rhoms}, suppose in addition that $\cc G=\cc E$, and for each closed point $s\in S$\quad  $\cc G|_s$ is stable and $\Ext^{3}(i_*\cc G,i_*\cc G)_0=0$ (cf. \eqref{trfr}). Then $\tau^{[1,2]} R\hom_p(i_*\cc G, i_* \cc G)$ fits in an exact triangle as follows 
$$\op{Cone}(f)[-1] \to \tau^{[1,2]} R\hom_p(i_*\cc G, i_* \cc G)\to R\hom_{p'}(\cc G, \cc G(\cc D))[-1].$$
\end{cor}
\begin{proof} By  \cite[Sec 4.4]{HT},  $\tau^{[1,2]} R\hom_p(i_*\cc G, i_* \cc G)$ is a 2-term complex of perfect amplitude $[1,2]$ which can be obtained as follows: first, $$\tau^{\ge 1} R\hom_p(i_*\cc G, i_* \cc G) \cong \op{Cone} \big(\O_S \xr{\id}   R\hom_p(i_*\cc G, i_* \cc G)\big),$$ the trace map \eqref{tr3} lifts to $$\tau^{\ge 1} R\hom_p(i_*\cc G, i_* \cc G) \to R^3p_* \O[-3],$$ and $\tau^{[1,2]} R\hom_p(i_*\cc G, i_* \cc G)$ is the co-cone of this map. We will now apply these two operations to the exact triangle in Lemma \ref{Rhoms} in the same order. First, by the stability of $\cc G|_s$, the $h^0$ part of degree 0 part of the exact triangle is  $$\hom_{p'}(\cc G,\cc G)\cong \hom_{p}(i_*\cc G,i_*\cc G) \cong \O_S.$$ Thus, applying $\tau^{\ge 1}$ we get an exact triangle as in the first row of the commutative diagram $$\xymatrix{\tau^{\ge 1} R\hom_{p'}(\cc G,  \cc G) \ar[r] \ar[d]^-{f} & \tau^{\ge 1} R\hom_p(i_*\cc G, i_* \cc G)\ar[r] \ar[d] & R\hom_{p'}(\cc G, \cc G(\cc D))[-1] \\ R^3p_*\O \ar@{=}[r] & R^3p_*\O.}$$ Now we get the result by taking co-cone over the rows in this diagram. 
\end{proof}

For simplicity, we will assume that the universal sheaf $\bb F$ of $X\times M^{\op s}$ exists \footnote{Without this assumption, what follows will still work after a basechange to an \'etale cover of $M^{\op s}$.}, and let $\bb D:=\div(\bb F)\into{i} X\times M^{\op s}$. Then $\bb F:=i_* \bb G$ where $\bb G$ is a sheaf on $\bb D$ flat over $ M^{\op s}$. 
\begin{prop} \label{long}
The tangent/obstruction sheaves of  the p.o.t \eqref{pot},  $$\ext^1_{p}(\bb F,\bb F), \quad \ext^2_{p}(\bb F,\bb F),$$ fit into the following exact sequence.
\ba 
0\to &\ext^1_{p'}(\bb G,\bb G)\to \ext^1_{p}(\bb F,\bb F)\to \hom_{p'}(\bb G,\bb G(\bb D))\to 
 \ext^2_{p'}(\bb G,\bb G) \to\\& \ext^2_{p}(\bb F,\bb F) \to \ext^1_{p'}(\bb G,\bb G(\bb D))\to \ext^3_{p'}(\bb G,\bb G)\to \O_{M^{\op s}}^{h^{0,3}(X)}\to 0.
\ea
\end{prop}
\begin{proof}
This is the cohomology long exact sequence of the exact triangle in Corollary \ref{conf}, by noting that $h^{k}\big(\op{Cone}(f)\big)\cong  \ext^k_{p'}(\bb G,\bb G)$ for $k=1,2$ and $$h^{3}\big(\op{Cone}(f)\big)\cong \ker\big(\ext^3_{p'}(\bb G,\bb G)\to R^3p_* \O\big).$$
\end{proof}

The complexes  $R\hom_{p'}(\bb G, \bb G)$,  $R\hom_{p'}(\bb G, \bb G(\bb D))$ are not necessarily perfect as $p'\colon \cc D\to M^{\op s}$ is not necessarily smooth. So they may not define elements of $K^{0}(M^{\op s})$, the K-group of locally free sheaves on $M^{\op s}$. However, Proposition \ref{long} shows that 
$$\tau^{[1,3]} R\hom_{p'}(\bb G, \bb G)-\tau^{[0,1]}R\hom_{p'}(\bb G, \bb G(\bb D))  \in K^{0}(M^{\op s}).$$

For any closed point $D \in |L|$, let $\bb G_D$ be the restriction of $\bb G$ to $$j\colon D\times M^{\op s}_D \into{} \bb D.$$ We denote the projection $D\times M^{\op s}_D\to M^{\op s}_D$ also by $p'$. By basechange property of $R\hom_{p'}$ and using the fact that $$Lj^*\big(\tau^{[1,2]} R\hom_p(i_*\bb G, i_* \bb G)\big)$$ is a perfect complex, after applying $Lj^*$ to the exact triangle in Corollary \ref{conf} and taking long exact sequence in cohomology as in Proposition \ref{long}, we see that 
$$ \tau^{[1,3]} R\hom_{p'}(\bb G_D, \bb G_D)-\tau^{[0,1]}R\hom_{p'}(\bb G_D, \bb G_D( D))  \in K^{0}(M^{\op s}_D).$$
Using the identification $M^{\op s}_D\cong M^{\op s}(D,P)$, this leads us to the following definition:
\begin{defn} \label{KGD} Suppose $D$ is an effective divisor in a nonsingular projective surface or Fano threefold $X$ and $P$ is a polynomial of degree $\dim D$. Let $\bb G_D$ be a universal sheaf over $D\times M^{\op s}(D,P)$.  Define a K-theory class in $K^0(M^{\op s}(D,P))$ by $$\sff K(\bb G_D):=\tau^{[1,3]} R\hom_{p'}(\bb G_D, \bb G_D)-\tau^{[0,1]}R\hom_{p'}(\bb G_D, \bb G_D(D)).$$ 
\end{defn}


\begin{cor} \label{GD} Let $\sff K(\bb G_D)$ be as in Definition \ref{KGD}. The virtual cycle of Proposition \ref{supp} is given by 
\ba [M^{\op s}(D,P)]^{\un \vir}=\big \{c\big(&\sff K(\bb G_D) \big) \cap c_F\big(M^{\op s}(D,P)\big)\big \}_{\vd_{M^{\op s}(D,P)}}.
\ea In particular, $[M^{\op s}(D,P)]^{\un \vir}$ only depends on scheme structure of $M^{\op s}(D,P)$, its universal sheaf $\bb G_D$, and the normal bundle $\O_D(D)$ of $D$ in $X$.
\end{cor} 
\begin{proof}
This follows from Proposition \ref{supp} and the exact sequence in Proposition \ref{long}. \end{proof}

The following Lemma can be used to simplify the formula in Corollary \ref{GD} in some cases. We will see an instance of this application in Section \ref{3d}.
\begin{lem} \label{normal}
\begin{enumerate}[i)]
\item Let $D$ be a normal Gorenstein surface, and $\cc G$ be a rank 1 torsion free sheaf on $D$. Then $\hom(\cc G, \cc G)\cong \O_D$.
\item Let $D$ be an integral Gorenstein curve,  and $\cc G$ be a rank 1 torsion free sheaf on $D$. Then $\hom(\cc G, \cc G)\cong \cc O_D(E)$ is a rank 1 reflexive sheaf given by an effective generalized divisor $E\subset D$ supported on singular locus of $D$.  
\end{enumerate}
\end{lem}

\begin{proof}
i) \quad Normality implies that singular locus of $D$, denoted by $S$, must be 0-dimensional or empty. 
Since $\cc G$ is torsion free it embeds into its double dual $\cc G^{**}$ with $\cc G^{**}/\cc G$ is 0-dimensional. By adjoint associativity 
$$\hom(\cc G, \cc G^{**})\cong \hom(\cc G \otimes \cc G^{*}, \O),$$ which shows that $\hom(\cc G, \cc G^{**})$ is reflexive (being the dual of a coherent sheaf). But on the nonsingular surface $D\setminus S$ we know  $\hom(\cc G, \cc G^{**})|_{D\setminus S} \cong \cc O_{D\setminus S}$ and two reflexive sheaves agreeing outside a codimension 2 subset are isomorphic \cite[Thm 1.12]{Ha}, so $\hom(\cc G, \cc G^{**}) \cong \cc O_{D}$. Now the natural  inclusion $\hom(\cc G, \cc G) \subseteq \hom(\cc G, \cc G^{**})$ proves that $\hom(\cc G, \cc G)$ is isomorphic to the ideal sheaf of some points on $D$. But,  since $\Hom(\cc G,\cc G)=\Gamma(\hom(\cc G, \cc G))\neq 0$ (in fact it is $\bb C$ by stability  of $\cc G$), this set of points must be empty i.e. $\hom(\cc G, \cc G)\cong \O_D$.

ii) \quad As in part i we see that $\hom(\cc G, \cc G^{**})$ is rank 1 and reflexive, the natural inclusion $\hom(\cc G, \cc G) \subseteq \hom(\cc G, \cc G^{**})$ is identity away from the singular locus of $D$, therefore by \cite[Cor 1.10]{Ha} $\hom(\cc G, \cc G)$ is reflexive. Finally, the natural section $\cc O_D\to \hom(\cc G, \cc G)$ (given by identity) shows that it is of the form $\O_D(E)$ as claimed.

\end{proof}

\section{Moduli spaces of 1-dimensional sheaves on surfaces} \label{2d}
In this section, we take $X$ to be a nonsingular projective surface with $H^1(\O_X)=0$.  The moduli space $M^{\op s}$  \eqref{mods} of stable 1-dimensional sheaves is equipped with the perfect obstruction theory \eqref{pot}. As before, for a given line bundle $L$ on $X$ let $M^{\op s}_L$ be the moduli space of stable 1-dimensional sheaves with fixed determinant $L$. Since $\pic(X)$ is 0-dimensional, $M^{\op s}_L\subset M^{\op s}$ is closed and open, and hence it inherits the same perfect obstruction theory as \eqref{pot}. We know $\Hom(\cc F,\cc F)\cong \bb C$ for any stable sheaf $\cc F \in M^{\op s}$ and also all $\Ext^{\ge 3}$ vanish, so the virtual dimension is 
\beq{vd}\vd_{M^{\op s}_L}=\op{ext}^1(\cc F,\cc F)-\op{ext}^2(\cc F,\cc F)=1-\chi(\cc F,\cc F)=1+L^2.\eeq  
Suppose $\dim |L|\neq 0$ and that $|L|$ is base point free.   Since $H^1(\O_X)=0$ we have $h^0(L|_D)=\dim |L|\neq 0$ for any $D\in |L|$, so there is a nonzero section $\O_D\to L|_D$, which induces a surjection $$H^1(\O_D)\onto{} H^1(L|_D),$$ and hence $h^0(\O_D)\ge h^1(L|_D)$. Therefore if in addition $D$ is irreducible so that ($h^0(\O_D)=1$), by \eqref{vd} $$\vd_{M^{\op s}_L}=1-\chi(\O_D,\O_D)=h^0(L|_D)-h^1(L|_D)+h^1(\O_D) \ge\dim |L|.$$  
This shows that the virtual dimension $$\vd_{M^{\op s}_D}=L^2+1-\dim |L|$$ of Proposition \ref{supp} is always nonnegative in the situation of this Section. 
\begin{prop}
Suppose $D\in |L|_{\op{ns}}$ then $[M^{\op s}_D]^{\un \vir}=[M^{\op s}_D]$ if $H^1(L|_D)=0$, and $[M^{\op s}_D]^{\un \vir}=0$, otherwise.
\end{prop}
\begin{proof} As we have seen before, $M^{\op s}_D\cong \jac(D)$ so is nonsingular of dimension $g:=h^{1}(\O_D)$, i.e. the genus of the nonsingular curve $D$. The universal sheaf $\bb G_D$ is identified with a Poincar\'e line bundle tensored by the pullback of a line bundle form $D$. In Corollary \ref{GD} the following identifications are then easily deduced \ba &\tau^{[1,3]} R\hom_{p'}(\bb G_D, \bb G_D)[1]\cong \ext^1_{p'}(\O_D,\O_D)\cong R^1p'_*{\O_D}\cong  T_{\jac(D)}\\&
\hom_{p'}(\bb G_D, \bb G_D(\bb D))\cong p'_* (L|_D), \quad \ext^1_{p'}(\bb G_D, \bb G_D(\bb D))\cong R^1p'_* (L|_D). \ea
Since $M^{\op s}_D$ is nonsingular its Fulton's Chern class is the same as the Chern class of its tangent bundle, and so we get  $$[M^{\op s}_D]^{\un \vir}=c_{h^1(L|_D)}(R^1p'_* (L|_D)) \cap [M^{\op s}_D].$$ Since $R^1p'_* (L|_D)\cong \O^{h^1(L|_D)}$, the claim follows. 
\end{proof}

If $K_X\cdot L<0$ then by adjunction formula and Serre duality for any $D\in |L|_{\op{ns}}$ we have $H^1(L|_D)\cong H^0(K_X|_D)=0$. In this situation we have
\begin{prop}\label{kxl}
Suppose $K_X\cdot L<0$ then $M^{\op s}_L$  is smooth and for any $D\in |L|$   
$$\sff c\big(K(\bb G_D)\big)=c\big(-\ext^1_{p'}(\bb G_D,\bb G_D)+\ext^2_{p'}(\bb G_D,\bb G_D)-\hom_{p'}(\bb G_D,\bb G_D(\bb D))\big),$$ (cf. Definition \ref{KGD} and Corollary \ref{GD}), and moreover,  $[M^{\op s}_D]^{\un \vir}=i_{|L|}^![M^{\op s}_L]$, where $[M^{\op s}_L]$ is understood as the sum of fundamental classes of connected components of $M^{\op s}_L$.  \end{prop}
\begin{proof}
 By Serre duality and stability of $\cc F$ $$\Ext^2(\cc F,\cc F)\cong  \Hom(\cc F,\cc F\otimes K)^*=0 \qquad \forall \cc F \in M^{\op s}.$$ 
By basechange the obstruction sheaf $\ext^2_{p}(\bb F,\bb F)=0$. Therefore, the tangent sheaf $\ext^1_{p}(\bb F,\bb F)$ is locally free and hence $M^{\op s}$ is smooth and so is its closed and open subset $M^{\op s}_L$ (note we are assuming $h^{0,1}(\O_X)=0$ in this Section). Therefore, $[M^{\op s}_L]^{\un \vir}=[M^{\op s}_L]$ and the formula follows by Propositions \ref{supp} and \ref{long}.   
\end{proof}
\subsection{Example: $X=\bb P^2$} \label{P2}
In this subsection we take $X=\bb P^2$ and $L=\O(d)$ for $d\ge 1$. This means that we are in the situation of Proposition \ref{kxl}, in particular $M^{\op {s}}(P)=M^{\op {s}}_L(P)$ is smooth of dimension $d^2+1$. The geometry of moduli space $M^{\op {ss}}_L(P)$ is studied by many people, e.g. see \cite{Le, CC} and the references within. It is shown in \cite{Le} that $M^{\op {ss}}_L(P)$ is irreducible and locally factorial. For $d \le 2$ $M^{\op {ss}}_L(P) \cong |L|$. For $d\ge 3$ its Picard group is free abelian with two generators $\rho^*\O_{|L|}(1)$ and  $\cc T_{\bb P^2}$ characterized by the following universal property: let $c$ be the constant term of the Hilbert polynomial $P(z)=dz+c$ and \beq{uu} u:=\frac{1}{\gcd(c,d)}(-d\O+c\O_\ell)\in K^0(\bb P^2),\eeq where $\ell \subset \bb P^2$ is a fixed line; any $S$-flat family $\cc F$ of semistable sheaves with Hilbert polynomial $P$ over $\bb P^2\times S$ determines the modular morphism $f\colon S\to M^{\op {ss}}_L(P)$, then $$\det( Rp_{2*}\;(\cc F\otimes p_1^*u))\cong f^*\cc T_{\bb P^2}$$ where $p_1, p_2$ are the projections to the first and second factors of $\bb P^2\times S$. 

For any $D\in |L|$ denote $\cc T_D:=\cc T_{\bb P^2}|_{M^{\op {ss}}_D(P)}$. If $D \in |L|_{\op{ns}}$ then $c_1(\cc T_D)$ is a multiple of the theta divisor on $M^{\op {ss}}_D(P)\cong \jac(D)$ (\cite{Le}); if $c=1$ it is $d$ times the theta divisor (\cite[Lem 2.8]{CC}). By preservation of interaction numbers this gives 
\begin{prop} \label{TD} Suppose $P(z)=dz+1$ then for any $D\in |\O_{\bb P^2}(d)|$
  $$c_1(\cc T_D)^g \cup [M^{\op s}_D]^{\un \vir} =d^gg!, \qquad g:=(d-1)(d-2)/2.$$ 
 \end{prop} \qed
 
It would be interesting to express the result of Proposition \ref{TD}  in case $D$ is non-reduced or reducible in terms of intersection numbers over the moduli spaces of stable sheaves on some divisors $\subset \bb P^2$ of degrees $<d$. In the rest of this Section we study this question in the following special case. Suppose $d=2d'>3$ and  $$D=2D'\in |\O_{\bb P^2}(d)|  \text{ for some nonsingular degree $d'$ curve } D' \into{j} \bb P^2.$$ The moduli space $M^{\op s}_D$ parameterizes two types of stable sheaves: rank 2 torsion free sheaves with scheme theoretic support $D'\subset D$ and also stable sheaves with scheme theoretic non-reduced support $D$. We still assume $P(z)=dz+1$, so in particular semistability implies stability and also there exists a universal sheaf. Let $Y$ be the total space of the line bundle $$q \colon \O_{D'}(d')\to D'.$$ $Y$ is a nonsingular quasi-projective variety with the zero section $D' \into{z} Y$ and $z^*K_Y\cong j^* K_{\bb P^2}$. Considering $D\hra Y$ as twice thickened zero section, we want to apply Proposition \ref{kxl} to the moduli space $M^{\op s}(Y,P)$ of stable compactly supported sheaves on $Y$ with Hilbert polynomial $P$ (with respect to $q^*\O_{D'}(1)$). To do this, we consider $Y$ as an open subset of the projective bundle $$\bb P:=\bb P(\O_{D'}\oplus \O_{D'}(d'))\xr{q} D'.$$
Let $B\subset |\O_{\bb P}(2)|$ be the open subset of divisors whose supports do not intersect the infinity divisor $\bb P\setminus Y$ (and so their supports remain compact in $Y$). Let $b\in B$ be the closed point corresponding to $D=2D'$ i.e. the twice thickened zero section $D'$ of $Y$. Let $M^{\op s}_{\O_{\bb P}(2)}(\bb P,P)$ be the moduli space of stable torsion sheaves on $\bb P$  with Hilbert polynomial $P$ (with respect to $q^*\O(1)$) and fixed determinant $\O_{\bb P}(2)$. Any closed point $\cc F \in M^{\op s}(Y,P)$ viewed as a sheaf in $\bb P$ has determinant $\O_{\bb P}(2)$, and hence we get the following cartesian diagram
$$ \xymatrix {M^{\op s}_D(P) \ar@{^(->}[r] \ar[d] & M^{\op s}(Y,P)\ar@{^(->}[r] \ar[d]^-{\div} & M^{\op s}_{\O_{\bb P}(2)}(\bb P,P) \ar[d]^-{\div}\\
\{b\} \ar@{^(->}[r]^-{i_B} &B\ar@{^(->}[r] &|\O_{\bb P}(2)|.}$$

To summarize, we view $M^{\op s}_D(P)$ as subschemes of both $M^{\op s}(Y,P)$ and $M^{\op s}_{\O_{\bb P^2}(d)}(\bb P^2,P)$ and since $D$ has the same normal bundle as subschemes of $Y$ and $\bb P^2$  the resulting virtual cycles are the same  (Corollary \ref{GD}) $$i_B^![M^{\op s}(Y,P)]^{\vir}=[M^{\op s}_D(P)]^{\un \vir}=i_{|\O_{\bb P^2}(d)|}^![M^{\op s}_{\O_{\bb P^2}(d)}(\bb P^2,P)]^{\vir}.$$
The line bundle $\cc T_D \to M^{\op s}_D(P)$ can also be  obtained by restricting a line bundle $\cc T_Y\to M^{\op s}(Y,P)$ defined analogously by replacing $u$ in \eqref{uu} by (recall that we have taken $c=1$ in this discussion) \beq{cc}-d\O+\O_f\in K(Y),\eeq where $f\subset Y$ is a fiber of $q$. 

The point of working on $Y$ (instead of $\bb P^2$) is that $Y$, being the total space of a line bundle, can be equipped with a $\bb C^*$-action scaling its fibers and with the fixed set $Y^{\bb C^*}=z(D')$. This induces a $\bb C^*$-actions on $B$ with $B^{\bb C^*}$=\{b\} and on $M^{\op s}(Y,P)$ and hence on its $\bb C^*$-invariant subset $M^{\op s}_D(P)\subset M^{\op s}(Y,P)$ with  $$M^{\op s}_D(P)^{\bb C^*}=M^{\op s}_L(Y,P)^{\bb C^*}.$$ Moreover, the morphism $\div \colon M^{\op s}(Y,P)\to B$ above is $\bb C^*$-equivariant.
The class \eqref{cc} has a lift \beq{ecc}-d\O+\O_f \otimes \bff t\eeq to the equivariant K-group $K^0(Y)^{\bb C^*}$, and so the line bundle $\cc T_Y$ and hence its restriction $\cc T_D$ have equivariant lifts. 

 Let $\bff t$ be the 1-dimensional representation of $\bb C^*$ determined by the above $\bb C^*$-action on each fiber of $Y$. The following facts about the components of the fixed sets can be proven similarly to e.g. \cite{GSY2, TT} and so we skip their proofs here. The fixed set  $M^{\op s}(Y,P)^{\bb C^*}$ has two types of components:  \subsubsection*{1. Rank 2 vector bundles over $D'$} The first type consists of sheaves of the form $z_*\cc G$ where $\cc G$ is a rank 2 stable vector bundle on $D'$. A component of $M^{\op s}(Y,P)^{\bb C^*}$ is therefore identified with $$M^{\op s}(D', 2,2g'-1),$$ the moduli space of rank 2 vector bundles of degree $2g'-1$ on $D'$, where $g':=(d'-1)(d'-2)/2$ is the genus of $D'$. This moduli space is projective (semistability implies stability because rank and degree are coprime) and nonsingular of dimension $4(g'-1)+1$ \cite{N}.  

\subsubsection*{2. Jacobian and symmetric products of $D'$} The second type of fixed sheaves consists of stable sheaves $\cc F$ fitting into a $\bb C^*$-equivariant short exact sequences $$0\to z_*\cc G_1\otimes \bff t^{-1} \to \cc F\to z_* \cc G_0\to 0$$ for some line bundles $\cc G_0, \cc G_1$ on $D'$. Any such sequence is split after pushing down via the affine morphism $q$, i.e. $q_*\cc F\cong \cc G_0 \oplus \cc G_1\otimes \bff t^{-1}$. The $\O_Y$-module structure on $\cc F$ then induces a nonzero section of the line bundle $\cc G_0^* \cc G_1 (d')$ (if zero, the short exact sequence above would be split and hence $\cc F$ would be unstable). $\cc F$ is thus determined by a line bundle $\cc G_1$ and an element of the Hilbert scheme or equivalently symmetric product on $D'$ $$\op{Sym}^{d'^2+\deg(\cc G_1)-\deg(\cc G_0)}(D').$$ 
 By the short exact sequence above and stability of $\cc F$ we must also have $$\deg(\cc G_1)+\deg(\cc G_0)=2g'-1, \qquad  \deg(\cc G_1)< (2g'-1)/2.$$ To summarize, there is a disjoint decomposition:
$$M^{\op s}_L(Y,P)^{\bb C^*}\cong M^{\op s}(D', 2,2g'-1) \cup \bigcup_{k=\lceil \frac{1-3d'}2 \rceil }^{g'-1} \pic^k(D')\times \op{Sym}^{2k+3d'-1}(D').$$

\subsubsection*{Restriction of p.o.t to $M^{\op s}(D', 2,2g'-1)$} Let $\bb G$ be the universal sheaf on $D'\times M^{\op s}(D', 2,2g'-1)$. We apply Proposition \ref{Rhoms} to the embedding  $$D'\times M^{\op s}(D', 2,2g'-1)\into{z} Y \times M^{\op s}(D', 2,2g'-1).$$ From the resulting exact triangle one easily deduce the exact sequence of vector bundles on $M^{\op s}(D', 2,2g'-1)$ 
$$0\to \ext^1_{p'}(\bb G,  \bb G)\to \ext^1_p(z_*\bb G, z_* \bb G) \to  \hom_{p'}(\bb G, \bb G(D')\otimes \bff t)\to 0.$$ The left term in this exact sequence is the tangent bundle of $M^{\op s}(D', 2,2g'-1)$  and is of rank $4(g'-1)+1$. The right term is the moving part of the restriction p.o.t.  to $M^{\op s}(D', 2,2g'-1)$ and is of rank $d^2+4(1-g')$.

\subsubsection*{Restriction of p.o.t to a second type component} We now turn to the component \beq{picsym} \cc C_k:=\pic^k(D')\times \op{Sym}^{2k+3d'-1}(D') \eeq of the fixed locus.  The universal sheaf $\bb F$ on $Y\times \cc C_k$ fits into an exact sequence of the form \beq{exsq} 0\to z_*\bb G_1\otimes \bff t^{-1} \to \bb F\to z_* \bb G_0\to 0,\eeq where $\bb G_1$ is a degree $k$ Poincar\'e line bundle pulled back from $D'$ times the first factor of $\cc C_k$, and $\bb G_0$ is a degree $2g'-1-k$ Poincar\'e line bundle pulled back from $D'$ times the second factor of $\cc C_k$,  and \beq{incz} z\colon D'\times \cc C_k \hra  Y \times \cc C_k \eeq is the inclusion.  
Using the exact sequence \eqref{exsq} and Proposition \ref{Rhoms} applied to the inclusion \eqref{incz}, one can see that the fixed part of obstruction theory is the tangent bundle of $\cc C_k$, and its moving part in K-theory is $$\cc \O^{2h^0(\O_{D'}(d'))}\otimes \bff t+p'_*(\bb G_1^*\bb G_0(d'))\otimes \bff t^2-p'_*(\bb G_1^*\bb G_0)\otimes \bff t-R^1p'_*(\bb G_0^*\bb G_1)\otimes \bff t^{-1}.$$

\subsubsection*{Restriction of $\cc T_D$ to the fixed components} The restriction of \eqref{ecc} to $D'$ is $u=-d\O_{D'}+\O_p\otimes \bff t$, and hence resrtriction of $\cc T_{D}$ to $M(D', 2,2g'-1)$ and $\cc C_k$ 
are respectively the equivariant line bundles $$\cc T_1:=\det p'_*(\bb G\otimes u),\qquad \cc T_{2,k}:=\det p'_*\big((\bb G_0+\bb G_1\otimes \bff t^{-1})\otimes u\big).$$

We are now ready to apply the virtual localization formula \eqref{GP} with $C=\{b\}$ (recall that $b\in B$ corresponds to the divisor $D$). We first find the weights of the action on $N_{C/B}\cong T_{B,b}$. Let $\bff s=c_1(\bff t)$. We use the identification $$H^0(\O_{\bb P}(2))\cong H^0(\O_{D'})\oplus H^0(\O_{D'}(d'))\otimes \bff t \oplus H^0(\O_{D'}(2d'))\otimes \bff t^2,$$ and $h^0(\O_{D'}(d'))=\frac 12d'(d'+3)$ and $h^0(\O_{D'}(2d'))=\frac 32d'(d'+1)$ to find that $$e(N_{C/B})=2^{\frac 32d'(d'+1)} \;\bff s^{d'(2d'+3)}.$$
 By \eqref{GP} 
\ba &2^{\frac {-3}2d'(d'+1)} [M_D]^{\un \vir}=\\& \bff s^{d'(2d'+3)}j'_* e\big(-\hom_{p'}(\bb G, \bb G(D')\otimes \bff t)\big)\cap [M(D',2,2g'-1)]\\
& +\bff s^{d'^2} \sum_{k=\lceil \frac{1-3d'}2 \rceil }^{g'-1} j'_* \frac{e\big(p'_*(\bb G_1^*\bb G_0)\otimes \bff t\big)\; e\big(R^1p'_*(\bb G_0^*\bb G_1)\otimes \bff t^{-1}\big)}{e\big(p'_*(\bb G_1^*\bb G_0(d'))\otimes \bff t^2\big)}\cap [\cc C_k].\ea 

Applying $c_1(\cc T_D)^g \cap -$ to both sides this formula and applying Proposition \ref{TD} we have proven
\begin{prop} \label{rk2D'}
\ba &2^{\frac {-3}2d'(d'+1)}d^gg!=\int_{M^{\op s}(D',2,2g'-1)}\frac{ \bff s^{d'(2d'+3)}c_1(\cc T_1)^g}{e\big(\hom_{p'}(\bb G, \bb G( D')\otimes \bff t)\big)}
\\
&+  \sum_{k=\lceil \frac{1-3d'}2 \rceil }^{g'-1} \int_{\cc C_k}\frac{\bff s^{d'^2}\; c_1(\cc T_{2,k})^g \;e\big(p'_*(\bb G_1^*\bb G_0)\otimes \bff t\big)\; e\big(R^1p'_*(\bb G_0^*\bb G_1)\otimes \bff t^{-1}\big)}{e\big(p'_*(\bb G_1^*\bb G_0(d'))\otimes \bff t^2\big)},
\ea where $\cc C_k$ is given in \eqref{picsym}.
\end{prop}\qed 

\section{Moduli spaces of 2-dimensional sheaves on threefolds} \label{3d}
In this section, we take $X$ to be a nonsingular projective threefold with $H^1(\O_X)=0=H^2(\O_X)$.  We assume that the Hilbert polynomial of $K_X$ is less than the Hilbert polynomial of $\O_X$ (so in particular $H^3(\O_X)=0$ by Serre duality), this ensures that Condition \eqref{trfr} is satisfied, and hence the moduli space $M^{\op s}=M^{\op s}(X,P)$  \eqref{mods} of stable 2-dimensional sheaves is equipped with the perfect obstruction theory \eqref{pot}. An example of this would be $X=\bb P^3$ that we will study later in this Section. As before, for a given line bundle $L$ on $X$ let $M^{\op s}_L$ be moduli space of stable 2-dimensional sheaves with fixed determinant $L$. Since $\pic(X)$ is 0-dimensional, $M^{\op s}_L\subset M^{\op s}$ is closed and open, and hence it inherits the same perfect obstruction theory. Unlike in Section \ref{2d} (Proposition \ref{kxl}), $M^{\op s}_L$ in this Section is usually obstructed, even if $L$ is very ample. We know $\Hom(\cc F,\cc F)\cong \bb C$ for any stable sheaf $\cc F \in M^{\op s}$ and also $\Ext^{3}(\cc F,\cc F)=0$ by Serre duality and stability of $\cc F$, so the virtual dimension is 
\beq{vd3}\vd_{M^{\op s}_L}=\op{ext}^1(\cc F,\cc F)-\op{ext}^2(\cc F,\cc F)=1-\chi(\cc F,\cc F)=1-L^2K_X/2.\eeq  
As in Section \ref{2d} the virtual dimension only depends on the first Chern class of $L$ (and not to the other parts of the Hilbert polynomial $P$). But in contrast, at least when $L$ is sufficiently positive (so that $|L|\neq \emptyset$ and $H^{i>0}(L)=0$) and $H^1(\O_D)=0$ for some $D\in |L|_{\op{ir}}$ we have  \beq{ineq}\vd_{M^{\op s}_L} \le \dim |L| \text{ with equality when $H^2(\O_D)=0$.} \eeq 

To see this claim, again since  $H^1(\O_X)=0$ we have $h^0(L|_D)=\dim |L|$, and also $H^{i>0}(L|_D)=0$ by vanishing of higher cohomologies of $\O_X$ and $L$, so by \eqref{vd3} $\vd_{M^{\op s}_L}=1-\chi(\O_D,\O_D)=h^0(L|_D)-h^2(\O_D).$

Donaldson-Thomas invariants of threefolds for which equality occurs in \eqref{ineq} were studied in \cite{GS1} and shown to have modular properties.

\begin{prop} \label{bsp}
Suppose $|L|$ is base point free, and $H^{i>0}(L)=0$. Also suppose that $H^{i>0}(\O_D)=0$ for some $D\in |L|_{\op{ns}}$. If $D\in |L|_{\op{ir}}$ is a normal surface then 
\ba [M^{\op s}_D]^{\un \vir}=\Big \{c\big(&\tau^{[1,3]} R\hom_{p'}(\bb G_D, \bb G_D)+\ext^1_{p'}(\bb G_D, \bb G_D( D)) \big) \cap c_F(M^{\op s}_D)\Big \}_{0},\ea
and $\deg [M^{\op s}_D]^{\un \vir}$ can be expressed as an integral over the Hilbert scheme of points over a nonsingular member of $|L|$.
\end{prop}
\begin{proof}
By assumptions and discussion above equality occurs in \eqref{ineq} for some $D$ (and hence for all $D\in |L|$), so $\vd_{M^{\op s}_D}=\vd_{M^{\op s}_L}-\dim |L|=0$. Now we apply Corollary \ref{GD} and note that $\hom_{p'}(\bb G_D, \bb G_D( D))$ is a trivial vector bundle by Lemma \ref{normal} and so has no effect in the Chern class. 

Now suppose $D\in |L|_{\op{ns}}$  be as in the statement. As we have seen $M^{\op s}_D$ is isomorphic to a Hilbert scheme of points on $D$ and hence is nonsingular.
Moreover, $\ext^{i\ge2}_{p'}(\bb G_D, \bb G_D)=0=\ext^{i\ge2}_{p'}(\bb G_D, \bb G_D( D))$ because of assumption $H^2(\O_D)=0=H^2(L|_D)$. This means $\ext^{1}_{p'}(\bb G_D, \bb G_D)$ and $\ext^{1}_{p'}(\bb G_D, \bb G_D( D))$ are vector bundles, the former being the tangent bundle of $M^{\op s}_D$ by the assumption $H^1(\O_D)=0$.
In particular, $c_F(M^{\op s}_D)=c(\ext^{1}_{p'}(\bb G_D, \bb G_D))\cap [M^{\op s}_D]$. Therefore the formula in Proposition simplifies to 
\ba [M^{\op s}_D]^{\un \vir}=\Big \{c\big(\ext^1_{p'}(\bb G_D, \bb G_D( D)) \big) \cap [M^{\op s}_D]\Big \}_{0}.\ea The last claim in Proposition now follows from the conservation of intersection numbers.
\end{proof}
\subsection{Example: $X=\bb P^3$}  \label{p3}
In this subsection we take $X=\bb P^3$ and $L=\O(d)$ for $d\ge 1$. $M^{\op {s}}(P)=M^{\op {s}}_L(P)$ has virtual dimension $2d^2+1$.  Of course, $|\O(d)|$ is base point free and $H^{i>0}(\O(d))=0$. For any $D$ hypersurface $H^1(\O_D)=0$ because $H^1(\O)=0=H^2(\O(-d))$. However, $H^2(\O_D)=0$ only if $d\le 3$, as $H^3(\O(-d))\neq 0$ otherwise. This means that for $d\le 3$ we are in the situation of Proposition \ref{bsp}
\subsubsection{\textbf{d= 2}} Fix the Hilbert polynomial $P(z)=z^2+2z+1-n$, where $n$ is a nonnegative integer.
Suppose $Q \in |\O(2)|$ is nonsingular. In this case, it is not hard to see that $M^{\op s}_Q(P)\cong Q^{[n]}$, the Hilbert scheme of $n$ points on the nonsingular quadratic surface $Q\cong \bb P^1\times \bb P^1$. Suppose $\bb I$ is the universal ideal sheaf on $Q\times Q^{[n]}$. By Proposition \ref{bsp}, \beq{degr} \deg [M^{\op s}_D(P)]^{\un \vir}=c_{2n} \big(\ext^1_{p'}(\bb I, \bb I(2)) \big) \cap [Q^{[n]}] \eeq for any $D\in |\O(2)|$.

Now suppose that $D=2D'\in |\O(2)|$ for some hyperplane $$\bb P^2 \cong D'\into{j} \bb P^3.$$ As in Section \ref{P2}, we would like to apply virtual localization to express $\deg [M^{\op s}_D(P)]^{\un \vir}$ as (virtual) integrations over $M^{\op s}(\bb P^2, P)$ and the nested Hilbert scheme over $\bb P^2$. Let $Y$ be the total space of the line bundle $$q \colon \O_{\bb P^2}(1)\to \bb P^2.$$ $Y$ is a nonsingular quasi-projective variety with the zero section identified with $D' \into{z} Y$ and we have $z^*K_Y\cong j^* K_{\bb P^3}$. By considering $D\hra Y$ as twice thickened zero section we can apply Proposition \ref{bsp} to the moduli space $M^{\op s}(Y,P)$ of stable compactly supported sheaves on $Y$ with Hilbert polynomial $P$ (with respect to polarization $q^*\O_{\bb P^2}(1)$). We can view $M^{\op s}_D(P)$ as a subscheme of $M^{\op s}(Y,P)$, and since $D$ has the same normal bundle as subschemes of $Y$ and $\bb P^2$ the resulting virtual cycles are the same  (Corollary \ref{GD}). 

$Y$ is equipped with a $\bb C^*$-action scaling the fibers of the line bundle and with the fixed set $Y^{\bb C^*}=z(D')$. Let $\bff t$ be the 1-dimensional representation of $\bb C^*$ determined by the above $\bb C^*$-action on each fiber of $Y$. The $\bb C^*$-action on $Y$ induces a $\bb C^*$-action on $M^{\op s}(Y,P)$. As in Section \ref{P2}, we have a decomposition of the fixed set (see \cite{GSY2, TT} for a proof)
$$M^{\op s}(Y,P)^{\bb C^*}\cong M^{\op s}(\bb P^2,P ) \cup \bigcup_{k=\lceil n/2\rceil }^{n}  (\bb P^2)^{[k,n-k]},$$ where $M^{\op s}(\bb P^2,P )$ is the moduli space of stable sheaves on $\bb P^2$ with Hilbert polynomial $P$, and $(\bb P^2)^{[n_1,n_2]}$ is the nested Hilbert scheme of 0-dimensional subschemes $Z_2\subseteq Z_1  \subset \bb P^2$ with $\op{length}(Z_i)=n_i$.   The component $M^{\op s}(\bb P^2,P )$ is a  nonsingular projective variety of dimension $4n-4$ parameterizing stable torsion free sheaves of rank 2 on $\bb P^2$ with Chern classes $c_1=-h$ (minus class of a line) and $c_2=n$ \cite{Ma}(Since degree and rank are coprime there no strictly semistable sheaves, explaining the projectivity.). Let $\bb G$ be the universal sheaf over $\bb P^2\times M^{\op s}(\bb P^2, P)$ and $\bb I_1\subseteq \bb I_2$ be the universal ideal sheave sheaves over $\bb P^2 \times (\bb P^2)^{[k,n-k]}$. 

The fixed and moving parts of the obstruction theory was worked out in \cite{GSY2, TT}. 
The fixed parts of the obstruction theory recovers the fundamental class $[M^{\op s}(\bb P^2,P )]$ and the virtual cycles$ [(\bb P^2)^{[k,n-k]}]^{\vir}$ constructed in \cite{GSY1, GT1, GT2} on the rest of the components. By \cite[Prop 3.2]{GSY2} the moving part of obstruction theory on the component $M^{\op s}(\bb P^2,P )$ is \beq{cN} \cc N=\hom_{p'}(\bb G, \bb G(1))\otimes \bff t-\ext^1_{p'}(\bb G, \bb G(1))\otimes \bff t.\eeq 
On the other hand, by  \cite[Prop 3.2]{GSY2} on the component  $(\bb P^2)^{[k,n-k]}$ the K-theory class of the moving part is 
\bal \label{Nk} &\cc N_{k}=R\hom_p(\bb I_1,\bb I_1(1))\otimes \bff t+R\hom_p(\bb I_2,\bb I_2(1))\otimes \bff t+R\hom_p(\bb I_2,\bb I_1(2))\otimes \bff t^2\\ \notag
&- R\hom_p(\bb I_1,\bb I_2(-1))\otimes \bff t^{-1}- R\hom_p(\bb I_2,\bb I_1(1))\otimes \bff t.\eal

As in Section \ref{P2}, there is a $\bb C^*$-equivariant morphism $\div\colon M^{\op s}(Y,P)\to B $ where $B$ is an open subset of a projective space with $B^{\bb C^*}$ a single point $b$ corresponding to the divisor $D$ (twice thickened zero section of $Y$). The weights of the action on $N_{C/B}\cong T_{B,b}$ needed in \eqref{GP} are determined  by the identity $$T_{B,b}\cong  H^0(\O_{D'}(1))\otimes \bff t \oplus H^0(\O_{D'}(2))\otimes \bff t^2,$$  from which we get $e(T_{B,b})=64 \bff s^9$.
By \eqref{GP} 
\ba &[M_D(P)]^{\un \vir}= 64s^9j'_* e(-\cc N)\cap [M^{\op s}(\bb P^2,P )]+\\
&64s^9  \sum_{k=\lceil n/2\rceil }^{n}   j'_* e(-\cc N_k) \cap [(\bb P^2)^{[k,n-k]}]^{\vir}.\ea 

If $\iota\colon (\bb P^2)^{[k,n-k]} \subset (\bb P^2)^{[k]}\times (\bb P^2)^{[n-k]}$ is the natural inclusion, it is shown in \cite{GSY1, GT1} that $$\iota_*[(\bb P^2)^{[k,n-k]}]^{\vir}=c_{n}(\ext^1_{p'}(\bb I_1,\bb I_2))\cap [(\bb P^2)^{[k]}\times (\bb P^2)^{[n-k]}].$$ 
By \eqref{degr} we have proven
\begin{prop} \label{p1p1} Let $\cc N$ and $\cc N_k$ be as in \eqref{cN} and \eqref{Nk} then
\ba & \int_{(\bb P^1\times \bb P^1)^{[n]}}c_{2n} \big(\ext^1_{p'}(\bb I, \bb I(2))\big)= \int_{M^{\op s}(\bb P^2,P )} \frac{ 64\bff s^{9}}{ e(\cc N)}
\quad +\\
& \sum_{k=\lceil n/2\rceil }^{n} \int_{(\bb P^2)^{[k]}\times (\bb P^2)^{[n-k]}}\frac{ 64\bff s^{9}\;c_{n}(\ext^1_{p'}(\bb I_1,\bb I_2))}{ e(\cc N_k)}.
\ea
\end{prop}\qed


\subsubsection{\textbf{d=3}} This case is similar to $d=2$. Fix the Hilbert polynomial $P(z)=3z^2/2+3z/2+1-n$, where $n$ is a nonnegative integer.
Suppose $Q \in |\O(3)|$ is nonsingular. Then,  $M^{\op s}_Q(P)\cong D^{[n]}$, the Hilbert scheme of $n$ points on the nonsingular cubic surface $Q$. Suppose $\bb I$ is the universal ideal sheaf on $Q\times Q^{[n]}$. By Proposition \ref{bsp}, $$\deg [M^{\op s}_D(P)]^{\un \vir}=c_{2n} \big(\ext^1_{p'}(\bb I, \bb I(3)) \big) \cap [Q^{[n]}]$$ holds for any $D\in |\O(3)|$.
It is possible to prove a similar formula as in Proposition \ref{p1p1} relating the intersection number $c_{2n} \big(\ext^1_{p'}(\bb I, \bb I(3)) \big) \cap [Q^{[n]}]$ to $$e\big(\ext^1_{p'}(\bb G, \bb G(1)) \otimes \bff t-\hom_{p'}(\bb G, \bb G(1))\otimes \bff t \big)\cap [M^{\op s}(\bb P^2,P )],$$ where $M^{\op s}(\bb P^2,P)$ is the moduli space of stable torsion free sheaves on $\bb P^2$ of rank 3 and Chern classes $c_1=-h$ (minus class of a line) and $c_2=n$. The correction terms will again consist of integrations over moduli spaces of flags of sheaves on $\bb P^2$. 

\subsubsection{\textbf{d=4 and beyond}} \label{d4} Fix the Hilbert polynomial $P(z)=2z^2+34-n$, where $n$ is a nonnegative integer.
Suppose $D \in |\O(4)|$ is nonsingular. Then, $M^{\op s}_D(P)\cong D^{[n]}$, the Hilbert scheme of $n$ points on the nonsingular quartic $K3$ surface. In this case $H^2(\O_D)\cong \bb C$ and so Proposition \ref{bsp} does not apply. 

Instead, let $\bb P^1\cong C \subset |\O(4)|$ be a line avoiding the locus of non-reduced or reducible divisors (having codimension $>1$ in $|\O(4)|$). By Lemma \ref{D/L} $M^{\op s}_C(P)\cong M^{\op s}(Y/\bb P^1,P)$, where $Y\subset \bb P^3\times \bb P^1$ is a nonsingular hypersurface of type (4,1) corresponding to  the Lefschetz pencil. It is a $K3$ fibration over $\bb P^1$ with finitely many fibers (108 of them) having a nodal singularity. Here,  any stable sheaf $\cc F\in  M^{\op s}_C(P))$ has rank 1 over its reduced irreducible support and so semistability implies stability.  Theorem \ref{rest}  
 gives a 0-dimensional cycle $[M^{\op s}_C(P)]^{\un \vir}  \in A_0(M^{\op s}_C(P))$. Similarly, one can consider a general hypersurface $Y$  of type $(4,a)$ in $\bb P^3\times \bb P^1$.  When $a=2$ $Y$ is a Calabi-Yau threefold. It would be interesting to find a relation between this $0$-dimensional cycles resulting from this construction and those studied in \cite{GS}.


Similarly for $d\ge 4$ one can choose $C$ to be a nonsingular subvariety of $\O(d)$ of dimension $h^2(\O_D)=(d-1)(d-2)(d-3)/6$, where $D\in |\O(d)|$, or more generally choose $C$ to be a nonsingular projective variety admitting a finite (and hence flat) morphism to a nonsingular $h^2(\O_D)$-dimensional subvariety of $|\O(d)|$. Should one choose a suitable Hilbert polynomial (so that semistability implies stability), Theorem \ref{rest}  
 will give a 0-dimensional cycle $$[M^{\op s}_C(P)]^{\un \vir}  \in A_0(M^{\op s}_C(P))$$ resulted from a perfect obstruction theory over an affine bundle over $M^{\op s}_C(P)$. We summarize discussion above in the following Proposition:
\begin{prop}\label{dge4}
Let $d\ge 4$ and $a \ge 1$ be two integers, $$N:=(d-1)(d-2)(d-3)/6,\qquad M:=2d^2+1-N,$$ and $Y$ be a general hypersurface in $\bb P^3\times \bb P^{N}$ of type $(d,a)$. Let $P$ be a degree 2 polynomial with leading coefficient $d/2$ such that there are no strictly semistable sheaves with Hilbert polynomial $P$ on the fibers of $Y/\bb P^N$. There is an affine bundle over $M^{\op s}(Y/\bb P^N,P)$ admitting a perfect obstruction theory resulting in a 0-dimensional cycle given by the formula 
\ba [M^{\op s}(Y/\bb P^N,P)]^{\un \vir}= &\big \{c\big(\tau^{[1,3]} R\hom_{p}(\bb G, \bb G)-\tau^{[0,1]}R\hom_{p}(\bb G, \bb G(d,a))\\&+\O_{\bb P^N}(a)^M \big)\cap c_F\big(M^{\op s}(Y/\bb P^N,P)\big)\big \}_{0}.\ea
Here, $\bb G$ is a universal (twisted) sheaf on $M^{\op s}(Y/\bb P^N,P)\times_{\bb P^N} Y$, and $\O(d,a):=\O_{\bb P^3}(d)\boxtimes \O_{\bb P^N}(a)$.
\end{prop}
\begin{proof} This is an application of Theorem \ref{rest} in combination with Lemma \ref{D/L}. The hypersurface $Y$ determines a modular morphism $i\colon \bb P^N\to |\O(d)|$, which is of degree $a$ onto its image, a linear subspace of $|\O(d)|$. The map $i$ is the composition of a finite flat morphism followed by a regular embedding with normal bundle $\O_{\bb P^N}(1)^M$. The rest of proof is similar to the argument leading to Corollary \ref{GD}.
\end{proof}

\vspace{.5cm}

\noindent \textit{amingh@umd.edu\\ Department of  Mathematics, University of Maryland,  College Park, MD 20742-4015.}



\begin{thebibliography}{10}
\bibitem[AIK]{AIK} Altman A, Iarrobino A, Kleiman S. Irreducibility of the compactified Jacobian. In Real and complex singularities (Proc. Ninth Nordic Summer School/NAVF Sympos. Math., Oslo, 1976) 1976 Aug 5 (Vol. 112).

\bibitem[BF]{BF} K.~Behrend, B.~Fantechi. The intrinsic normal cone. Inventiones mathematicae. 128(1): 45--88 (1997).
\bibitem[CC]{CC} J.~Choi, K.~Chung. The geometry of the moduli space of one-dimensional sheaves. Science China Mathematics. 58(3): 487--500 (2015).
  \bibitem[CO]{CO} E.~Carlsson, A.~Okounkov, Exts and vertex operators. Duke Math. J. 161: 1797--1815 (2012).         
\bibitem[EG]{EG} D.~Edidin, W.~Graham. Equivariant intersection theory (With an appendix by A,~ Vistoli: The Chow ring of $M_2$). Inventiones mathematicae. 131(3): 595--634 (1998).
\bibitem[E]{E} D.~Eisenbud. Commutative Algebra: with a view toward algebraic geometry. Springer Science \& Business Media (2013).
\bibitem[F1]{F1} J.~Fogarty. Fixed point schemes. American Journal of Mathematics. 5(1): 35--51 (1973).
\bibitem[F]{F} J.~Fogarty. Truncated Hilbert Functors. Reine u. Ang. Math. 234: 65--88 (1969).
\bibitem[Fu]{Fu} W.~Fulton. Intersection theory. Springer-Verlag (1998).
\bibitem[FG]{FG} B.~Fantechi, L.~G\"ottsche L. Riemann-Roch theorems and elliptic genus for virtually smooth schemes. Geometry \& Topology. 14(1): 83--115 (2010).
\bibitem[GS1]{GS1} A.~Gholampour, A.~Sheshmani. Donaldson-Thomas invariants, linear systems and punctual Hilbert schemes. \arXiv{1909.02679}.
\bibitem[GS]{GS} A.~Gholampour, A.~Sheshmani. Donaldson-Thomas invariants of 2-dimensional sheaves inside threefolds and modular forms. Advances in Mathematics. 326: 79--107 (2018).
\bibitem[GSY1]{GSY1} A.~Gholampour, A.~Sheshmani and S.-T.~Yau. Nested Hilbert schemes on surfaces: Virtual fundamental class. \arXiv{1701.08899}.
\bibitem[GSY2]{GSY2} A.~Gholampour, A.~Sheshmani and S.-T.~Yau. Localized Donaldson-Thomas theory of surfaces. \arXiv{1701.08902}.
\bibitem[GT1]{GT1} A.~Gholampour and R.~P.~Thomas. Degeneracy loci, virtual cycles and nested Hilbert schemes. \arXiv{1709.06105}.
\bibitem[GT2]{GT2} A.~Gholampour and R.~P.~Thomas. Degeneracy loci, virtual cycles and nested Hilbert schemes. \arXiv{1902.04128}.
\bibitem[GP]{GP} T.~Graber, R.~Pandharipande. Localization of virtual classes. Inventiones mathematicae. 135(2):487--518 (1999).
\bibitem[Ha]{Ha} R.~Hartshorne. Generalized divisors on Gorenstein schemes. K-theory. 8(3): 287--339 (1994).
\bibitem[HL]{HL} D.~Huybrechts, M.~Lehn. The geometry of moduli spaces of sheaves. Cambridge University Press (2010).
\bibitem[HT]{HT} D.~Huybrechts, R.~P.~Thomas. Deformation-obstruction theory for complexes via Atiyah and Kodaira-Spencer classes. Mathematische Annalen. 346(3): 545--569 (2010).
\bibitem[J]{J} J.~P.~Jouanolou. Une suite exacte de Mayer-Vietoris en K-th\'eorie alg\'ebrique. Higher K-theories, Springer, Berlin, Heidelberg:  293--316 (1973).
\bibitem[KL]{KL} Y.~H.~Kiem, J.~Li. Localizing virtual cycles by cosections. Journal of the American Mathematical Society. 26(4): 1025--50 (2013).
\bibitem[KM]{KM} F.~Knudsen, D.~Mumford. The projectivity  of the moduli space of stable curves I: preliminaries on Det and Div. Mathematica Scandinavica. 39(1): 19--55 (1977).
\bibitem[Kr]{Kr} A.~Kresch, Cycle groups for Artin stacks, Invent. Math. 13: 495--536 (1999). 
\bibitem[Le]{Le}  J.~ Le Potier. Faisceaux semi-stables de dimension 1 sur le plan projectif. Rev. Roumanine Math. Appl., 38: 635--678 (1993).
\bibitem[Ma]{Ma} M.~Maruyama. Moduli of stable sheaves, II. Journal of Mathematics of Kyoto University. 18(3): 557--614 (1978).
\bibitem[M]{M} H.~Matsumura. Commutative ring theory. Cambridge university press (1989).
\bibitem[MPT]{MPT} D.~Maulik, R.~Pandharipande, R.~P.~Thomas. Curves on K 3 surfaces and modular forms. Journal of Topology. 3(4): 937--996 (2010).
\bibitem[Mu]{Mu} D.~Mumford. Lectures on Curves on an Algebraic Surface. (AM-59). Princeton University Press (2016).
\bibitem[N]{N} P.~E.~Newstead. Introduction to moduli problems and orbit spaces. TIFR Lect. Notes. 51 (1978).
\bibitem[S]{S} B.~Siebert. Virtual fundamental classes, global normal cones and Fulton's canonical classes. In Frobenius manifolds. Vieweg + Teubner Verlag: 341--358 (2004). 
\bibitem[TT]{TT} Y.~Tanaka and R.~P.~Thomas.Vafa-Witten invariants for projective surfaces I: stable case. \arXiv{1702.08486}.
\bibitem[T]{T} R.~P.~Thomas. A holomorphic Casson invariant for Calabi-Yau 3-folds, and bundles on $K3$ fibrations. J. Differential Geom. 54: 367--438 (2000). 
              
\bibitem[Z]{Z} Zheng X. The Hilbert schemes of points on surfaces with rational double point singularities. \arXiv{1701.02435}. 
 \end{thebibliography}
\end{document}